\documentclass[12pt]{amsart}
\usepackage[utf8]{inputenc} 
\usepackage[T1]{fontenc}
\usepackage{color}
\usepackage{comment}
\usepackage{times}
\usepackage[hidelinks]{hyperref}
\usepackage{enumerate,latexsym}
\usepackage{latexsym}
\usepackage{subcaption}

\usepackage{fullpage}

\usepackage{amsmath,amsthm,amsfonts,amssymb}
\usepackage{graphicx}
\usepackage{dsfont}

\usepackage{tikz}
\usetikzlibrary{arrows.meta}
\usetikzlibrary{knots}
\usetikzlibrary{hobby}
\usetikzlibrary{arrows,decorations.markings}

\usetikzlibrary{fadings}

\makeatletter
\def\namedlabel#1#2{\begingroup
 #2%
 \def\@currentlabel{#2}%
 \phantomsection\label{#1}\endgroup
}
\makeatother

\usepackage[lite]{amsrefs}

\renewcommand{\PrintDOI}[1]{\href{http://dx.doi.org/\detokenize{#1}}{doi: \detokenize{#1}}%
	\IfEmptyBibField{pages}{, (to appear in print)}{}}

\theoremstyle{plain}
\newtheorem*{theorem*}{Theorem}
\newtheorem*{thmex*}{Theorem~\ref{example}}
\newtheorem*{thmasymp*}{Theorem~\ref{thmAsymp}}
\newtheorem{theorem}{Theorem}[section]

\newtheorem{lemma}[theorem]{Lemma}

\newtheorem{proposition}[theorem]{Proposition}

\newtheorem{example}[theorem]{Example}

\theoremstyle{definition}
\newtheorem{definition}[theorem]{Definition}

\newcommand{\ben}{\begin{enumerate}}
\newcommand{\een}{\end{enumerate}}

\newcommand{\ed}{\end{document}}

\definecolor{rrr}{rgb}{.9,0,.1}

\definecolor{rr}{rgb}{.8,0,.3}

\newcommand{\tr}{\triangleright}

\usepackage{pdfsync}
\graphicspath{ {./images/} }

\title[A Polynomial Invariant of Stuck Knots]{Generalized Quandle Polynomials and Their Applications to Stuquandles, Stuck Links, and RNA Folding}
\author[E. Bondarenko]{Ekaterina  Bondarenko}
\address{Hamilton College, Clinton, NY, USA}
\email{ebondare@hamilton.edu}
\author[J. Ceniceros]{Jose Ceniceros}
\address{Hamilton College, Clinton, NY, USA}
\email{jcenicer@hamilton.edu}

\author[M. Elhamadi]{Mohamed Elhamdadi}
\address{University of South Florida, Tampa, Florida, USA}
\email{emohamed@usf.edu}
\author[B. Jones]{Brooke Jones}
\address{University of South Florida, Tampa, Florida, USA}
\email{brookejones1@usf.edu}

\begin{document}

\maketitle

\begin{abstract}
We introduce a generalization of the quandle polynomial.  We prove that our polynomial is an invariant of stuquandles.  Furthermore, we use the invariant of stuquandles to define a polynomial invariant of stuck links. As a byproduct, we obtain a polynomial invariant of RNA foldings. Lastly, we provide explicit computations of our polynomial invariant for both stuck links and RNA foldings.
\end{abstract}
\parbox{5.5in} {\textsc{Keywords:} quandles, quandle polynomial, stuquandles, stuquandle polynomial, stuck links, arc diagrams, RNA folding

                \smallskip
                
                }

\section{Introduction}\label{intro}
Stuck links can be considered as a generalization of singular links.  They were introduced in \cite{B}.  Stuck links are physical links where the strands are stuck together in a fixed position, with one strand above the other. As a result, their projections have two types of crossings: the classical crossings from classical knot theory and a different type of crossing called a stuck crossing. Stuck knots and links have applications to RNA folding through a transformation relating stuck link diagrams to arc diagrams of an RNA folding, see \cite{B,CEKL, CEMR}. Additionally, the use of this generalization of classical knot theory is uniquely equipped to model both the entanglement and intra-chain interactions of a biomolecule as described in \cite{B}. 

In \cite{CEKL}, a generating set of the oriented stuck Reidemeister
moves for oriented stuck links was introduced.  The generating set of oriented stuck Reidemeister moves was used to define an algebraic structure called \emph{stuquandle}. The motivation of the stuquandle algebraic structure was to axiomatize the oriented stuck Reidemeister moves, thus allowing the construction of the fundamental stuquandle associated with a given stuck link. Using the fundamental quandle, the coloring counting invariant of stuck links was defined. As a consequence, the coloring counting invariant for arc diagrams of RNA foldings was constructed through the use of
stuck link diagrams. The coloring counting invariant of stuck links is defined as the cardinality of the set of homomorphisms from the fundamental stuquandle to a finite stuquandle.
Although the stuquandle counting invariant is a useful invariant of stuck links, it is not strong enough, and thus, we define an enhancement of it in this article. In the case of racks and quandles, the study of enhancements of the counting invariant is a very active area of research, see \cite{CNS, CJKLS, CES, CEGS}. 

Specifically, in \cite{N}, a two-variable polynomial from finite quandles encodes a set with multiplicities arising from counting trivial actions of elements on other elements of the quandle. This polynomial was used to define a polynomial invariant of classical links and was shown to be an enhancement of the quandle coloring counting invariant. Additionally, the quandle polynomial invariant was extended to the case of singular knots in \cite{CCE1}. In this article, we generalize these polynomials to the case of stuck links. We then use these polynomials to define the an enhancement of the coloring counting invariant of stuck links and of RNA foldings. Our approach in this article is different than the enhancement of the stuquandle counting invariant in \cite{CEMR}, which was achieved by assigning Boltzmann weights at both classical and stuck crossings and thus leading to a single-variable, a two-variable and a three-variable polynomial invariant of stuck links and applied to arc diagrams of RNA foldings.

 This article is organized as follows.  In Section~\ref{RSK}, we review the basics of stuck knots and their diagrammatics. In Section~\ref{sec3}, we recall the relationship between stuck links and arc diagrams.  Specifically, we review the transformation to obtain a stuck link diagram from an arc diagram and vice versa.  In Section~\ref{sec4}, we discuss the algebraic structures motivated by the diagrammatic representation of stuck knots and the fundamental stuquandle, leading to the stuquandle counting invariant.
  Section~\ref{quandle polynomial} reviews the definition of the quandle polynomial, the subquandle polynomial, and the link invariants obtained from the subquandle polynomial.  A generalization of the quandle polynomial is introduced in Section~\ref{sec6}.  We end this section by proving that this generalization is an invariant of stuquandles and then use the generalization to define a polynomial invariant of stuck links.  Lastly, in Section~\ref{Examples}, we provide explicit computations of our invariants for both stuck links and RNA foldings.  In the case of RNA foldings, we give an example of two arc diagrams that are not distinguished by the stuquandle counting invariant but are distinguished by the substuquandle polynomial invariant.

\section{Review of Stuck Knots and Links}\label{RSK}
Stuck links generalize singular links and were introduced in \cite{B}. Furthermore, stuck links are useful for modeling biomolecules, particularly RNA folding, as discussed in \cite{B,CEKL,CEMR}. This article follows the definitions and conventions introduced in\cite{B}.

A stuck link diagram may include classical and stuck crossings. A stuck crossing is a singular crossing with additional information about the stuck position. Figure \ref{SX} shows a singular crossing, while Figure \ref{StuckX} illustrates the two types of stuck crossings.
\begin{figure}[h!]
    \centering
    \includegraphics[scale=.25]{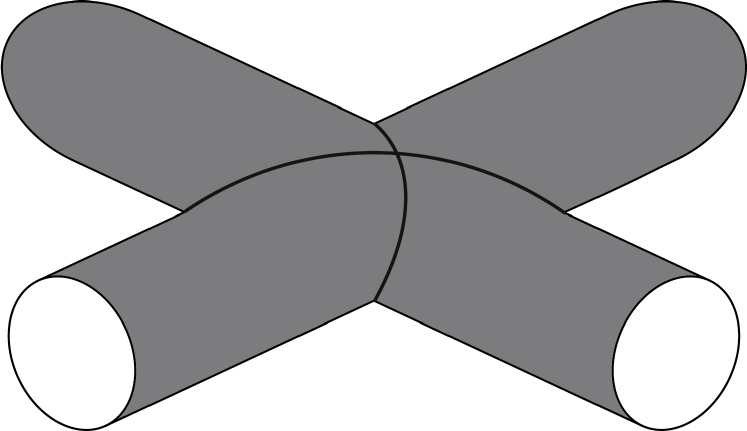}\hspace{2cm}
    \includegraphics[scale=1]{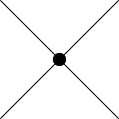}
    \caption{Singular crossing in a singular link (left) and a singular crossing in a singular link diagram (right).}
    \label{SX}
\end{figure}

\begin{figure}[h!]
    \centering
    \includegraphics[scale=.25]{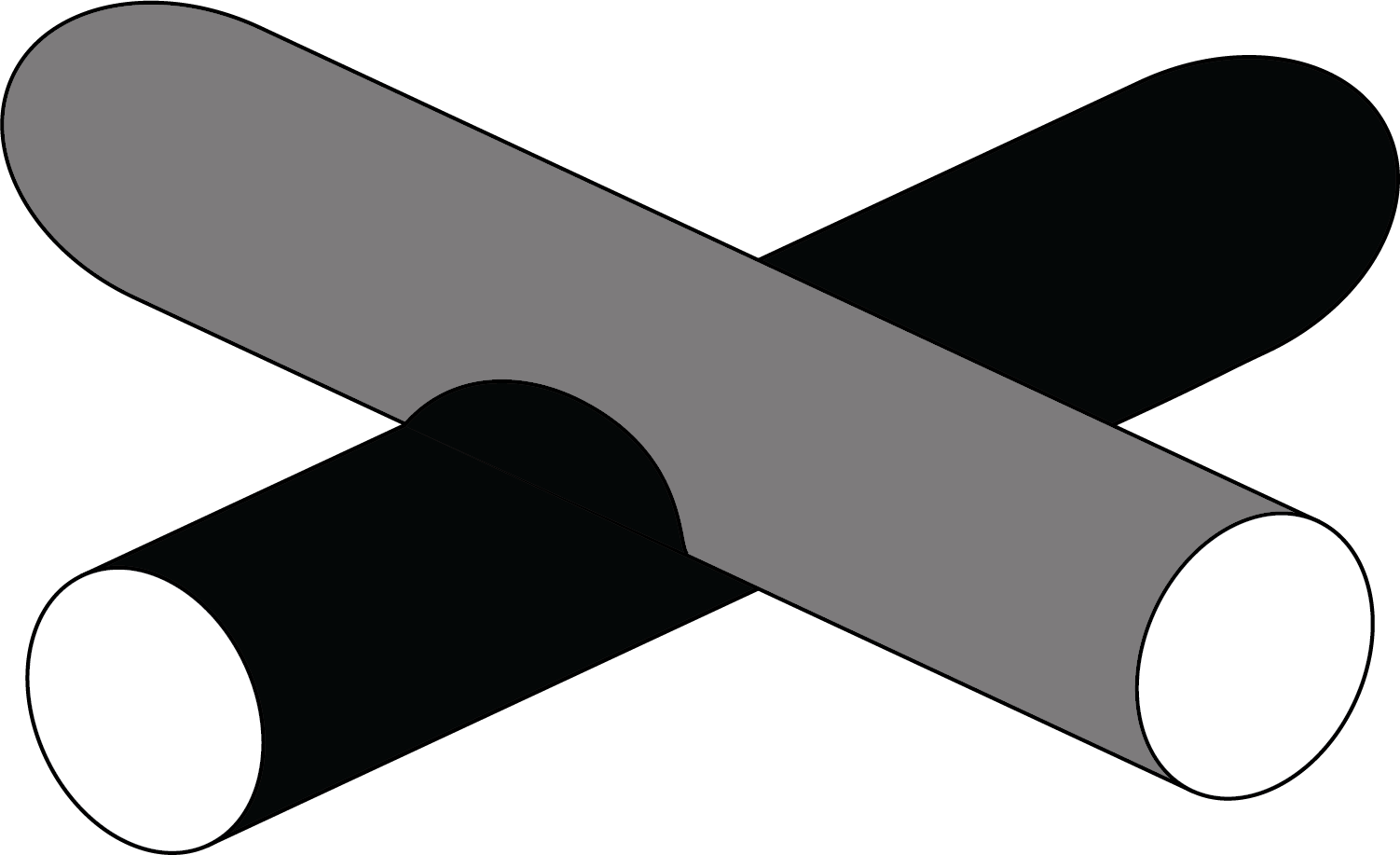}\hspace{2cm}
    \includegraphics[scale=1]{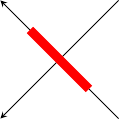}

\vspace{.5cm}
    \includegraphics[scale=.25]{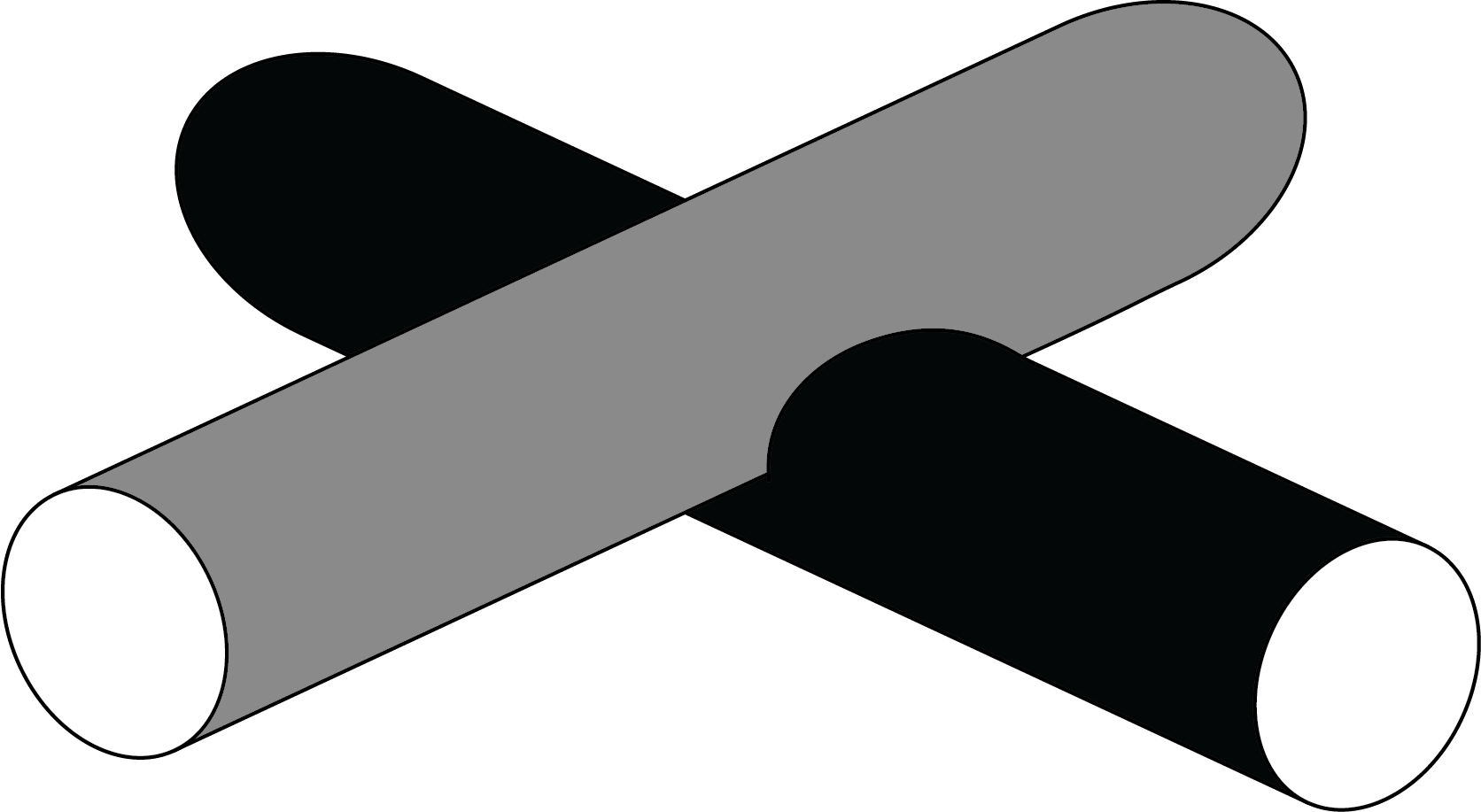}\hspace{2cm}
    \includegraphics[scale=1]{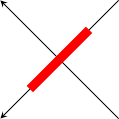}
\caption{Stuck crossings in a stuck link (left) and stuck crossings in a stuck link diagram (right).}
    \label{StuckX}
\end{figure}
We will use a thick bar on the over arc at a stuck crossing in a stuck link diagram; see Figure \ref{StuckX}. Following the right-hand rule, we will refer to the top crossings in Figure \ref{StuckX} as positive stuck crossings, while the bottom crossings will be negative stuck crossings.

\begin{figure}[h!]
    \centering
\includegraphics[scale=1]{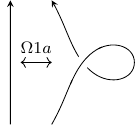} \hspace{2cm}
\includegraphics{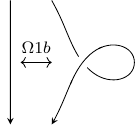}

\vspace{.5cm}
\includegraphics{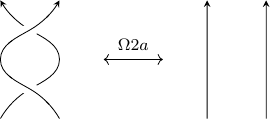}

\vspace{.5cm}
\includegraphics{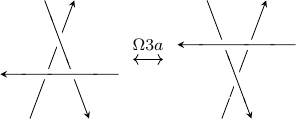}

\vspace{.5cm}
\includegraphics{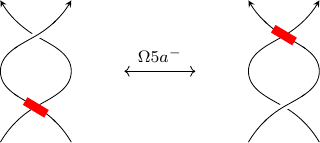}\hspace{2cm}
\includegraphics{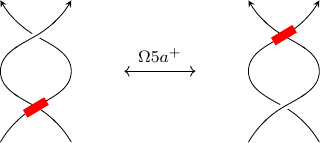}

\vspace{.5cm}
\includegraphics{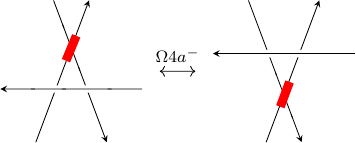}\hspace{2cm}
\includegraphics{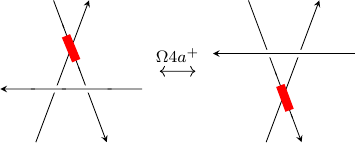}

\vspace{.5cm}
\includegraphics{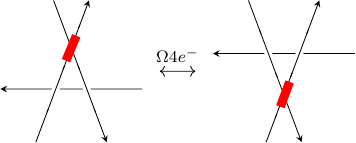}\hspace{2cm}
\includegraphics{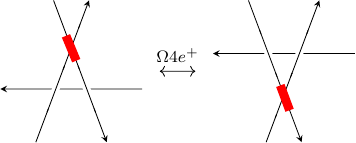}
    \caption{A generating set of oriented stuck Reidemeister moves.}
    \label{stuckrmoves}
\end{figure}

In \cite{CEKL}, the set of moves shown in Figure~\ref{stuckrmoves} was identified as a generating set of oriented stuck Reidemeister moves. We will follow the naming convention established in that paper. The oriented stuck Reidemeister moves are essential for studying stuck links through the use of stuck link diagrams. Two stuck link diagrams are equivalent if one can be transformed into the other by a finite sequence of planar isotopies and oriented stuck Reidemeister moves. Thus, similar to classical links, a \emph{stuck link}  is defined as an equivalence class of stuck link diagrams modulo the oriented stuck Reidemeister moves.

\section{Stuck links and Arc Diagrams}\label{sec3}
In this section, we review arc diagrams and the relationship between stuck links and arc diagrams. In \cite{KM}, Kauffmann and Magarshak introduced arc diagrams as a combinatorial way of studying the topology of RNA folding. The motivation of arc diagrams, as noted in \cite{KM}, is that the RNA molecule is a long chain consisting of the bases A (adenine), C (cytosine), U (uracil), and G (guanine). The pairs (A and U) and (C and G) are able to form bonds with each other. Therefore, an RNA molecule may be represented as a linear sequence of the letters A, C, U, and G, and folding of the molecule is a possible pairing of a given sequence of bases for a complete description and theory of arc diagrams see \cite{KM}. The following example is taken from \cite{KM}. First, consider the chain $\cdots CCCAAAACCCCCUUUUCCC\cdots$ with the folding given in Figure~\ref{RNAchain}.
\begin{figure}[ht]
    \centering
    \includegraphics[scale=.5]{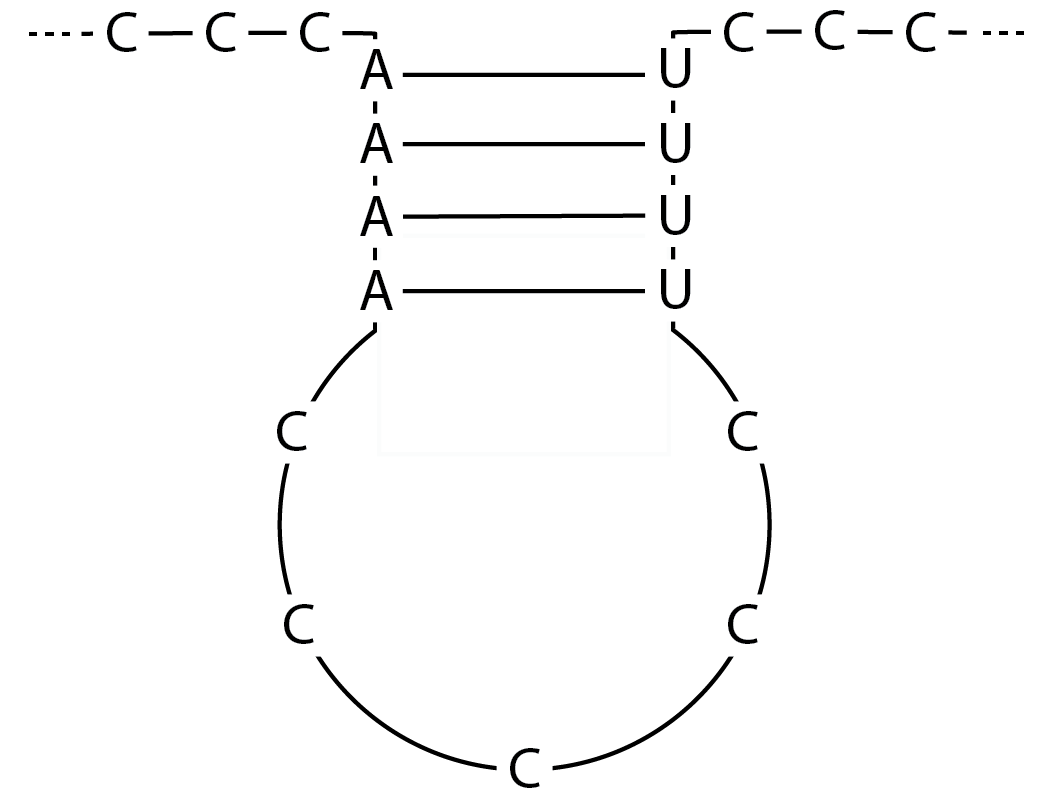}
    \caption{An example of an RNA folding.}
    \label{RNAchain}
\end{figure}
\\

In Figure~\ref{RNAfolding}, the diagram on the left is an abstraction of the folding introduced by Kauffmann and Magarshak, called an \emph{arc diagram}. In our work, we will use the convention from \cite{B} by further abstracting the diagram. We will replace the four connecting arcs with just one solid gray stripe, as shown on the right in Figure~\ref{RNAfolding}.

\begin{figure}[ht]
    \centering
    \includegraphics[scale=.5]{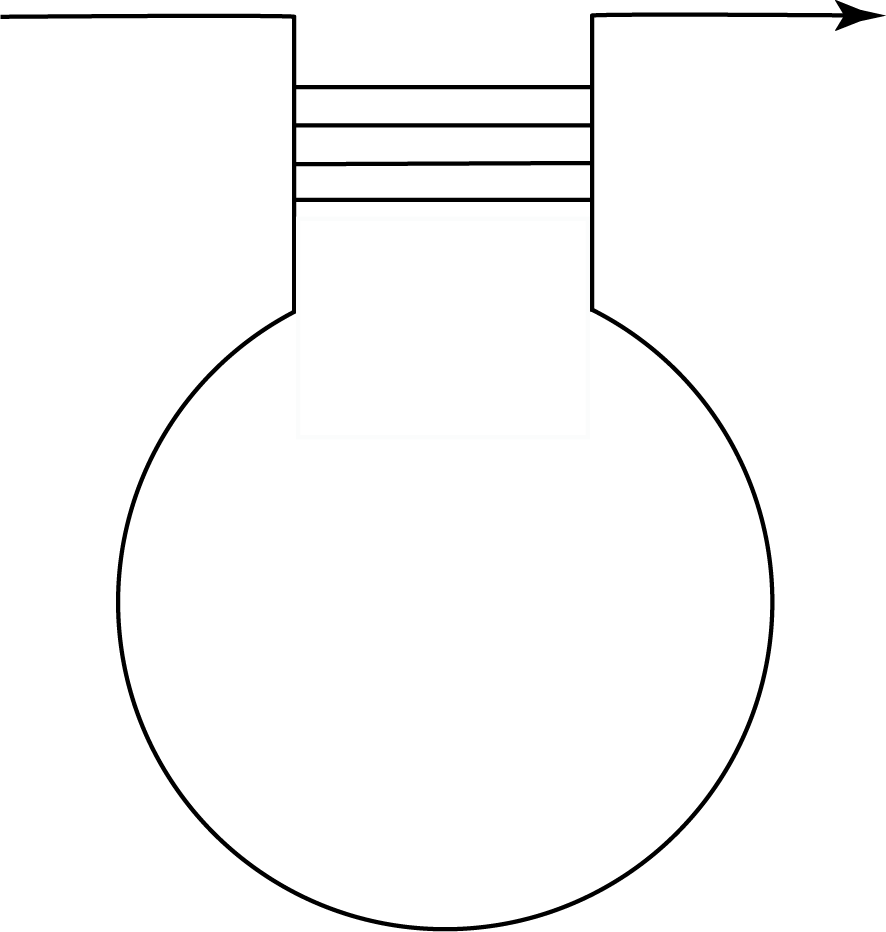}
    \hspace{1cm}
    \includegraphics[scale=.5]{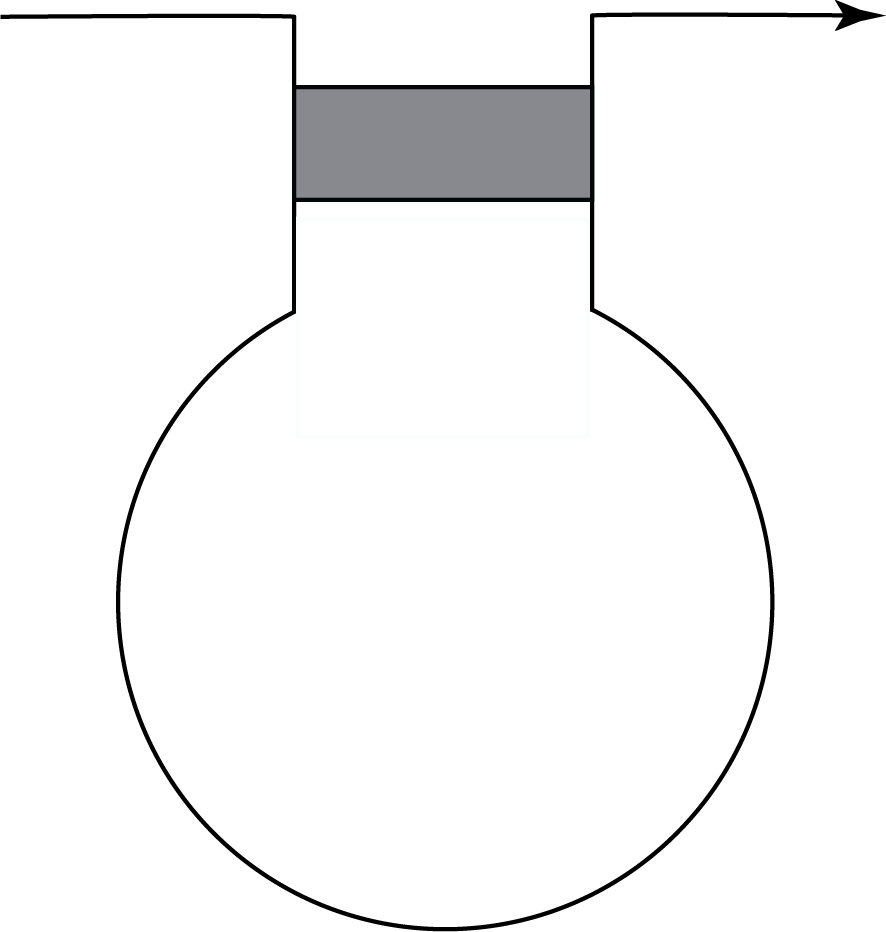}
    \caption{Arc diagram of RNA folding (left) and an arc diagram of RNA folding with gray stripe (right). }
    \label{RNAfolding}
\end{figure}
 In \cite{B}, the following transformation is defined and makes the connection between stuck links and arc diagrams. In order to change from an arc diagram of RNA folding to a stuck link, we apply the transformation, $T$, see Figure~\ref{TRNAfolding1}. The same transformation may be used when starting with a negative stuck crossing. See Figure~\ref{TRNAfolding2}.  
\begin{figure}[ht]
    \centering
    \includegraphics[scale=.5]{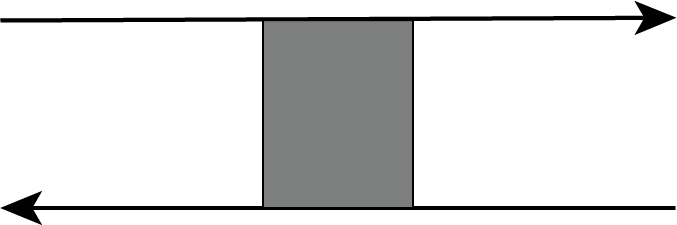}
    \hspace{.6cm}
    \includegraphics[scale=1]{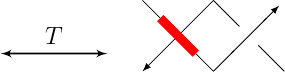}
    \includegraphics[scale=1]{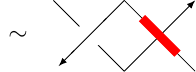}

    \vspace{1.5cm}
    \includegraphics[scale=.5]{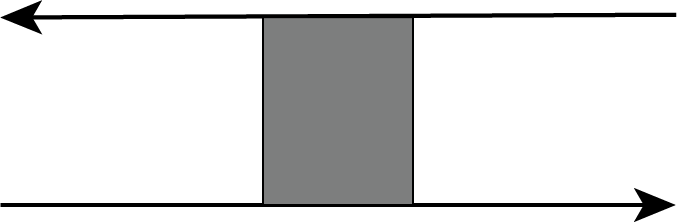}
    \hspace{.6cm}
    \includegraphics[scale=1]{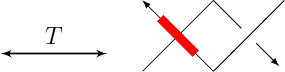}
    \includegraphics[scale=1]{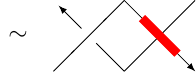}
    \caption{The transformation $T$.}
    \label{TRNAfolding1}
\end{figure}

\begin{figure}[ht]
    \centering
    \includegraphics[scale=1]{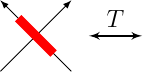}
\hspace{.2cm}
    \includegraphics[scale=.5]{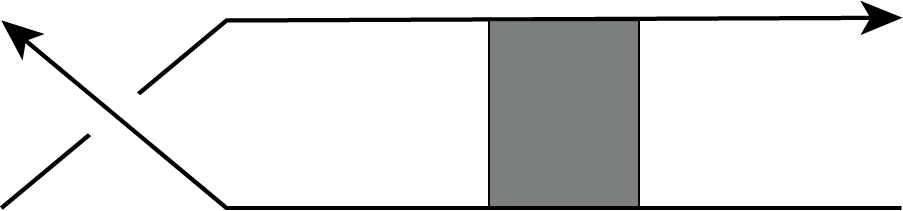}
    \put(10,10){$\sim$}
    \includegraphics[scale=.5]{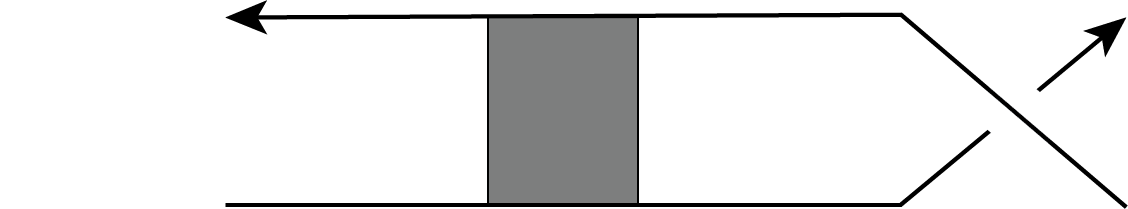}
    \vspace{1.5cm}
    
\includegraphics[scale=1]{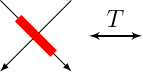}
\hspace{.2cm}    
    \includegraphics[scale=.5]{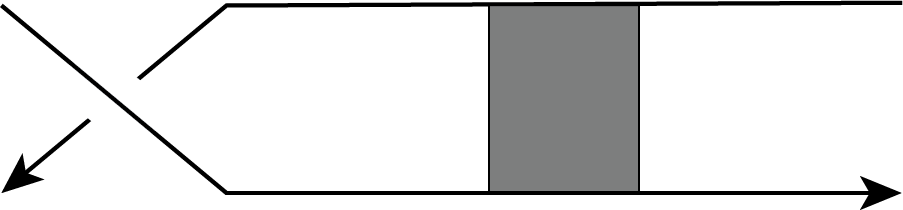}
        \put(10,10){$\sim$}
    \includegraphics[scale=.5]{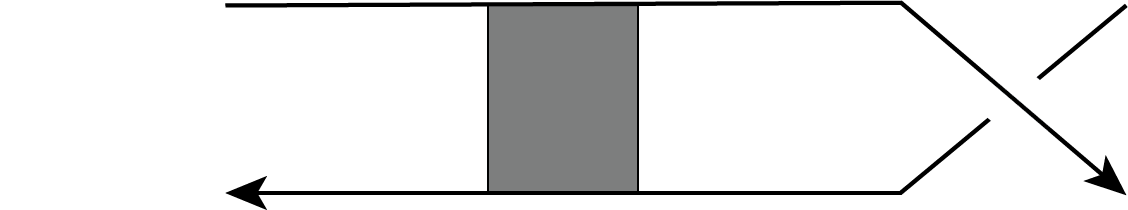}
\caption{Transformation starting with a stuck crossing.}
    \label{TRNAfolding2}
\end{figure}

In the following example, we will consider an arc diagram of an RNA folding and apply the transformation $T$ to obtain a stuck link diagram. Since the strands in an arc diagram can be closed in several ways, we will avoid any ambiguity by following the approach in \cites{B, TLKL} of self-closure. By the self-closure of the arc diagram, we mean that each strand of the arc diagram gets connected to itself. 

\begin{example}\label{RNAFoldingStuck}
In this example, we consider an arc diagram of an RNA folding with two strands, see Figure~\ref{example}. First, since the arc diagram has two strands, the self-closure means that we connect the loose ends of one strand to each other and the loose ends of the other strand to each other. Next,  we replace the gray stripe by applying $T$, see the top figure in Figure~\ref{TRNAfolding1}. 
\begin{figure}
    \centering
\includegraphics[scale=.4]{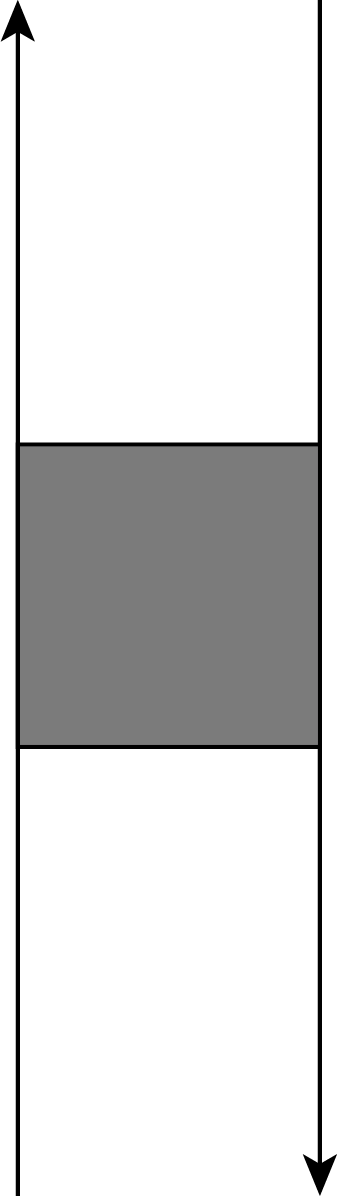}
\hspace{1cm}
\includegraphics[scale=1]{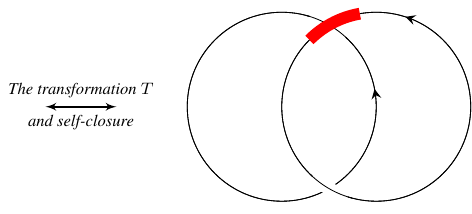}
    \caption{Arc diagram and corresponding stuck link diagram.}
    \label{example}
\end{figure}
\end{example}

The transformation, $T$, will play a key role in allowing us to define an invariant of arc diagrams from the polynomial invariant of stuck links. We also note that in order to study the topology of RNA folding in three-dimensional space through the use of an arc diagram, a set of Reidemeister-type moves was introduced in \cite{KM}. Therefore, an RNA folding is an equivalence class of arc diagrams modulo the Reidemeister-type moves. For more information on the Reidemeister-type moves allowed on an arc diagram and details about arc diagrams, please refer to \cite{KM}.

\section{Algebraic Structures from Stuck Knots}\label{sec4}

 In this section, we discuss the algebraic structures motivated by the diagrammatics of stuck knots and links. For more details on quandles, singquandles, and stuquandles the reader is referred to \cite{EN,BEHY, CEKL}.
 
 The following definition is motivated by the  Reidemeister moves in classical knot theory.  
\begin{definition}
    A \textit{quandle} is a set $X$ with a binary operation $*:X\times X\rightarrow X$ satisfying the following three axioms:
    \begin{itemize}

\item[(i)]({\it right distributivity}) for all $x,y,z\in X$, we have $(x* y)* z=(x* z)* (y* z)$;
\item[(ii)] ({\it invertibility}) for all $x\in X$, the map $R_x:X\rightarrow X$ sending $y$ to $y* x$ is a bijection;
\item[(iii)] ({\it idempotency}) for all $x\in X$, $x* x=x.$
\end{itemize}
If $S\subset X$ is itself a quandle, we call $S$ a \textit{subquandle} of $X$. 
\end{definition}
In the rest of the article, for all $x,y \in X$, we will denote ${R_y}^{-1}(x)$ by $x \overline{*}y$. The next definition is motivated by the generalized Reidemeister moves in singular knot theory.  
\begin{definition}\label{singqdle}
    Let $(X,*)$ be a quandle and $R_1$ and $R_2$ be maps from $X\times X$ to $X$. The quadruple $(X,*,R_1,R_2)$ is an \textit{oriented singquandle} if for all $x,y,z \in X$:
	\begin{eqnarray}
		R_1(x\bar{*}y,z)*y&=&R_1(x,z*y) \label{eq1} \\
		R_2(x\bar{*}y, z) & =&  R_2(x,z*y)\bar{*}y \label{eq2}\\
	      (y\bar{*}R_1(x,z))*x   &=& (y*R_2(x,z))\bar{*}z \label{eq3} \\
R_2(x,y)&=&R_1(y,x*y)  \label{eq4} \\
R_1(x,y)*R_2(x,y)&=&R_2(y,x*y) \label{eq5}.   	
\end{eqnarray}	
\end{definition}

\begin{definition}\label{stuquandle}
	Let $(X, *,R_1,R_2 )$ be a singquandle and $R_3$ and $R_4$ be maps from $X \times X$ to $X$. The 6-tuple $(X, *, R_1, R_2, R_3, R_4)$ is called an \emph{oriented stuquandle} if the following axioms are satisfied for all $x,y,z \in X$:
	\begin{eqnarray}
		R_3(y,x)*R_4(y,x)&=&R_4(x*y,y) \label{eq6},\\
	    R_4(y,x)&=&R_3(x*y,y) \label{eq7},\\
	    R_3(y*x,z)& =&R_3(y,z\bar{*}x)*x  \label{eq8} ,\\
		R_4(y,z\bar{*}x)&=&R_4(y*x,z)\bar{*}x, \label{eq9}\\
	    (x*R_4(y,z))\bar{*}y &=& (x\bar{*}R_3(y,z))*z. \label{eq10}
\end{eqnarray}	
\end{definition}

Let $(X, *, R_1, R_2, R_3, R_4)$ be a stuquandle, and let $L$ be a stuck link with diagram $D$. A coloring of $D$ by $X$ is an assignment of elements of $X$ to the semiarcs at stuck crossings and to the arcs at classical crossings of $D$, obeying the coloring rules in Figure~\ref{newrule}. 

\begin{figure}[ht]
\centering
\vspace{.5cm}
\includegraphics{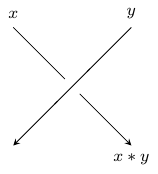}\hspace{2cm}
\includegraphics{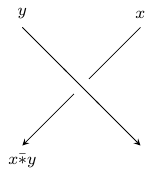}

\vspace{.5cm}
\includegraphics{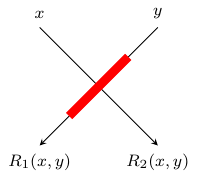}\hspace{2cm}
\includegraphics{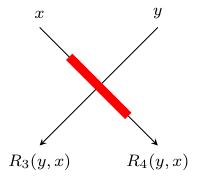}
\caption{Coloring rules at classical and stuck crossings.}
\label{newrule}
\end{figure}
We note that the stuquandle axioms correspond to the oriented stuck Reidemeister moves, following the coloring rules in Figure~\ref{newrule}. We will now review some key concepts about stuquandles, including basic examples.

\begin{definition}
    Let $(X, *, R_1, R_2, R_3, R_4)$ be a stuquandle. A subset $S\subset X$ is called a \textit{substuquandle} if $(S, *, R_1, R_2, R_3, R_4)$ is itself a stuquandle. 
\end{definition}

\begin{definition}\cite{CEKL} \label{homo}
Let $(X,*,R_1,R_2,R_3,R_4)$ and $(Y,\triangleright,S_1,S_2,S_3,S_4)$ be two stuquandles. A map $f:X\longrightarrow Y$ is a \textit{stuquandle homomorphism} if the following conditions are satisfied:
\begin{eqnarray}
f(x*y)&=&f(x)\triangleright f(y),\label{homomorphism1}\\
f(R_1(x,y))&=&S_1(f(x),f(y)),\\
f(R_2(x,y))&=&S_2(f(x),f(y)),\\
f(R_3(x,y))&=&S_3(f(x),f(y)),\\
f(R_4(x,y))&=&S_4(f(x),f(y)).\label{homomorphismlast}
\end{eqnarray}

If $f$ a bijective stuquandle homomorphism then it is called a \textit{stuquandle isomorphism}. 
\end{definition}

\begin{lemma}\label{image}
Let $(X,*,R_1,R_2,R_3,R_4)$ and $(Y,\triangleright,S_1,S_2,S_3,S_4)$ be two stuquandles. If \(f: X \rightarrow Y\) is a stuquandle homomorphism, then the image of \(f\), denoted by \(\text{Im}(f)\), is a substuquandle of \(Y\).
\end{lemma}

\begin{proof}
By definition, the image of \( f \) is \( \text{Im}(f) = \{ f(x) \mid x \in X \} \subseteq Y \).
Since \( f \) is a stuquandle homomorphism, it preserves the stuquandle operations and maps. Specifically, for any \( x, y \in X \) the equations \ref{homomorphism1}-\ref{homomorphismlast} from Definition~\ref{homo} are satisfied.
Therefore, for any \( f(x), f(y) \in \text{Im}(f) \), the elements $ f(x) \triangleright f(y), S_1(f(x), f(y)), S_2(f(x), f(y))$,
\noindent
$ S_3(f(x), f(y)), S_4(f(x), f(y))$ are also in \( \text{Im}(f) \). Hence, \( (\text{Im}(f), \triangleright, S_1, S_2, S_3, S_4) \) satisfies all the stuquandle axioms and is a substuquandle of \( Y \). 
\end{proof}

\begin{example}\label{generalizedaffinestuquandle}\cite{CEKL}
Let $X=\mathbb{Z}_n$ with the quandle operation $x*y = ax+(1-a)y$, where $a$ is invertible so that $x\ \bar{*}\ y = a^{-1}x+(1-a^{-1})y$.  Now let $R_1(x,y) = bx+cy$, then by axiom~(\ref{eq4}) of Definition~\ref{singqdle} we have $R_2(x,y) = acx + [c(1-a) + b]y$. By substituting these expressions into axiom~(\ref{eq1}) of Definition~\ref{singqdle} we find the relation $c= 1 - b$. Substituting, we find that the following is an oriented singquandle for any invertible $a$ and any $b$ in $\mathbb{Z}_n$:
	\begin{eqnarray}
    	x*y &=& ax + (1-a)y, \label{sing1} \\
        R_1(x,y) &=& bx + (1 - b)y, \label{sing2} \\
        R_2(x,y) &=& a(1 - b)x + (1 - a(1 - b))y. \label{sing3}
    \end{eqnarray}
    Next, let $R_3(x,y) = dx+ey$, then by axiom~(\ref{eq7}) of Definition~\ref{stuquandle}, we have $R_4(x,y) =  [d(1-a) + e]x+ady$.  By substituting these expressions into axiom~(\ref{eq6}) of Definition~\ref{stuquandle} we obtain no constraints on the coefficients $d$ and $e$.
    Axiom~(\ref{eq8}) of Definition~\ref{stuquandle} imposes that $d=1-e$, while axioms~(\ref{eq9}) and (\ref{eq10}) of Definition~\ref{stuquandle} introduce no constraint. 
    Substituting, we obtain
\begin{eqnarray}
  R_3(x,y)&=&(1-e)x+ey,\\
  R_4(x,y)&=&(1-a(1-e))x+a(1-e)y.
\end{eqnarray}
 Therefore, $(\mathbb{Z}_n,*,R_1,R_2,R_3,R_4)$ is an oriented stuquandle where $a$ is invertible element of $\mathbb{Z}_n$ and any $b,e \in \mathbb{Z}_n$.

\end{example}

\begin{example}\cite{CEKL}
Let $\Lambda=\mathbb{Z}[t^{\pm 1},v]$  and let $X$ be a $\Lambda$-module. Let
\begin{eqnarray*}
x\ast y &=& tx+(1-t)y,\\
R_1(x,y) &=& \alpha(a,b,c)x+(1-\alpha(a,b,c))y,\\
R_2(x,y) &=& t[1  - \alpha(a,b,c)] x + [1 -t(1-  \alpha(a,b,c))] y,\\
R_3(x,y) &=& (1-\alpha(d,e,f))x+\alpha(d,e,f)y,\\
R_4(x,y) &=&  [1 -t(1-  \alpha(d,e,f))] x + t[1  - \alpha(d,e,f)] y ,\\
\end{eqnarray*}
where $\alpha(a,b,c)=at+bv+ctv$ and $\alpha(d,e,f)=dt+fv+etv$. Then $(X,*,R_1,R_2,R_3,R_4)$ is an oriented stuquandle, which we call an \textit{Alexander oriented stuquandle}. The fact that $(X,*,R_1,R_2,R_3,R_4)$ is an oriented stuquandle follows from Example \ref{generalizedaffinestuquandle} by straightforward substitution. 
\end{example}
Given a stuck link $L$ with diagram $D$, the notion of its fundamental stuquandle \((\mathcal{STQ}(L), \ast, R_1, R_2, R_3, R_4)\) was introduced in \cite{CEKL}. For completeness we include the definition of the fundamental stuquandle of a stuck link.

\begin{definition}\cite{CEKL}
Let $D$ be a stuck link diagram of a stuck link $L$ and let $S=\{ a_1,a_2,\dots, a_m\}$ be the set of labels of the arcs in $D$ at classical crossings and semiarcs in $D$ at stuck crossings. We define the \emph{fundamental stuquandle} of $D$ by proceeding as in classical knot theory with the fundamental quandle.

\begin{enumerate}
    \item The set of stuquandle words, $W(S)$, is recursively defined. 
    \begin{enumerate}
        \item $S \subset W(S)$,
        \item If $a_i, a_j \in W(S)$, then 
    \[ a_i*a_j, a_i \bar{*} a_j, R_1(a_i,a_j), R_2(a_i,a_j), R_3(a_i,a_j), R_4(a_i,a_j) \in W(S). \]
    \end{enumerate}
    \item The set $Y$ is the set of \emph{free stuquandle words} which are equivalent classes of $W(S)$ determined by the conditions in Definition~\ref{stuquandle}.
    \item Let $c_1,\dots, c_n$ be the crossings of $D$. Each crossing $c_i$ in $D$ determines a relation $r_i$ on the elements of $Y$.
    \item The \emph{fundamental stuquandle} of $D$, $\mathcal{STQ}(D)$, is the set of equivalence class of words in $W(S)$ determined by the stuquandle conditions and the relations given by the crossings of $D$.
\end{enumerate}
\end{definition}

In the following example, we use the notation and naming convention of stuck knots and links from \cite{CEKL}. 

\begin{example}
In this example, we compute the fundamental stuquandle of the following oriented stuck link.  Consider the stuck trefoil, denoted by $2_1^{k-}$, with one negative stuck crossing and two negative classical crossings; see Figure~\ref{trefoil}. 
\begin{figure}[h]

\includegraphics{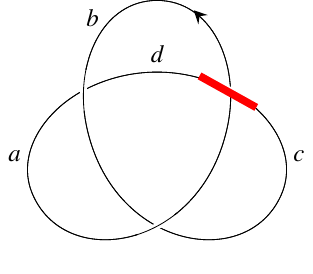}
    \caption{Diagram $D$ of the stuck trefoil $2^{k-}_1$.}
    \label{trefoil}
    
\end{figure}

We will label the arcs of the diagram $D$ by $a,b,c,$ and $d$.  Then the fundamental stuquandle of $2^{k-}_1$ is defined by

\begin{eqnarray*}
\mathcal{STQ}(2^{k-}_1)&\cong&\langle a,b,c,d \, \vert \, \; a=d\, \bar{*}\,b, \; b=R_3(a,c),\; c=b\,\bar{*}\,a,\; d=R_4(a,c)\rangle.
\end{eqnarray*}

\end{example}

The fundamental stuquandle can be used to define the following computable and effective invariant of stuck links. Given a finite stuquandle \((X, \tr, R'_1, R'_2, R'_3, R'_4)\), the set of stuquandle homomorphisms from \(\mathcal{STQ}(L)\) to \(X\), denoted by \(\textup{Hom}(\mathcal{STQ}(L), X)\), is used to define computable invariants. Specifically, by computing the cardinality \(\vert\textup{Hom}(\mathcal{STQ}(L), X) \vert\), we obtain an integer value invariant called the \emph{stuquandle counting invariant}, denoted by \(Col_X(L)\). 

Each \(f \in \textup{Hom}(\mathcal{STQ}(L), X)\) can be thought of as assigning an element of \(X\) to each arc in $D$ at a classical crossing and to each semiarc in $D$ at each stuck crossing, satisfying the coloring rules in Figure~\ref{newrule}. Therefore, each \(f \in \textup{Hom}(\mathcal{STQ}(L), X)\) can be represented by the \(m\)-tuple \( \left(f(a_1), f(a_2), \dots, f(a_m) \right)\), where \(a_1, a_2, \dots, a_m\) are the arc labels of any diagram of $D$. Furthermore, the image of each element of $\textup{Hom}(\mathcal{STQ}(L), X)$ is a substuquandle of $X$ by Lemma~\ref{image}.

\section{Review of the Quandle Polynomial}\label{quandle polynomial}
In this section, we recall the definition of the quandle polynomial, the subquandle polynomial, and the link invariants obtained from the subquandle polynomial. For a detailed construction of these polynomials, see \cites{EN,N}.
\begin{definition}
Let  $(Q,*)$ be a finite quandle. For any element $x \in Q$,  let
\[ C(x) = \lbrace y \in Q \, : \, y *x = y \rbrace \quad \text{and} \quad R(x) = \lbrace y \in Q \, : \, x *y = x \rbrace  \]
and set $r(x) = \vert R(x) \vert$ and $c(x) = \vert C(x) \vert$. Then the \emph{quandle polynomial of Q}, $qp_Q(s,t)$, is 
\[ qp_Q(s,t) = \sum_{x \in Q} s^{r(x)}t^{c(x)}.\]
\end{definition}

In \cite{N}, the quandle polynomial was shown to be an effective invariant of finite quandles. In addition to being an invariant of finite quandles, the quandle polynomial was generalized to give information about how a subquandle is embedded in a quandle.

\begin{definition}
Let $S \subset Q$ be a subquandle of $Q$. The \emph{subquandle polynomial of S}, $qp_{S \subset Q}(s,t)$, is
\[qp_{S \subset Q}(s,t) = \sum_{x \in S} s^{r(x)}t^{c(x)}\]
where $r(x)$ and $c(x)$ are defined above.
\end{definition}
Note that for any knot or link $K$, there is an associated fundamental quandle, $Q(K)$, and for any given finite quandle $T$ the set of quandle homomorphism, denoted by $\textup{Hom}(Q(K),T)$, has been used to define computable link invariants, for example, the cardinality of the set is known as the \emph{quandle counting invariant}. In \cite{N}, the subquandle polynomial of the image of each homomorphism was used to enhance the counting invariant.

\begin{definition}
Let $K$ be a link and $T$ be a finite quandle. Then for every $f \in \textup{Hom}(Q(K),T)$, the image of $f$ is a subquandle of $T$. The \emph{subquandle polynomial invariant}, $\Phi_{qp}(K,T)$, is the set with multiplicities
\[ \Phi_{qp}(K,T) = \lbrace qp_{\textup{Im}(f) \subset T} (s,t) \, \vert \, f \in \textup{Hom}(Q(K),T) \rbrace.\]
Alternatively, the multiset can be represented in polynomial form by
\[  \phi_{qp}(K,T) = \sum_{f \in \textup{Hom}(Q(K),T)} u^{qp_{\textup{Im}(f) \subset T} (s,t)}.\]
\end{definition}

\section{Generalized Quandle Polynomials}\label{sec6}

In this section, we introduce a generalization of the quandle polynomial in \cite{N}. We show that this generalization is an invariant of stuquandles. We then use the generalization to define a polynomial invariant of stuck links.  

\begin{definition}
    Let $(X,*,R_1,R_2,R_3,R_4)$ be a finite stuquandle. For every $x\in X$, define \begin{align*}
        C^1(x)&=\{y\in X|y*x=y\} \hspace{.3cm}\text{ and }\hspace{.3cm} R^1(x)=\{y\in X|x*y=x\},\\
        C^2(x)&=\{y\in X|R_1(y,x)=y\} \hspace{.3cm}\text{ and }\hspace{.3cm} R^2(x)=\{y\in X|R_1(x,y)=x\},\\
        C^3(x)&=\{y\in X|R_2(y,x)=y\} \hspace{.3cm}\text{ and }\hspace{.3cm} R^3(x)=\{y\in X|R_2(x,y)=x\},\\
        C^4(x)&=\{y\in X|R_3(y,x)=y\} \hspace{.3cm}\text{ and }\hspace{.3cm} R^4(x)=\{y\in X|R_3(x,y)=x\},\\
        C^5(x)&=\{y\in X|R_4(y,x)=y\} \hspace{.3cm}\text{ and }\hspace{.3cm} R^5(x)=\{y\in X|R_4(x,y)=x\}.\\
    \end{align*}
   Let $c^i(x)=|C^i(x)|$ and $r^i(x)=|R^i(x)|$ for $i=1,2,3,4,5$. Then the \textit{stuquandle polynomial of $X$} is 
   \[ stqp(X)=\sum_{x\in X} s_1^{r^1(x)}t_1^{c^1(x)}s_2^{r^2(x)}t_2^{c^2(x)}s_3^{r^3(x)}t_3^{c^3(x)}s_4^{r^4(x)}t_4^{c^4(x)}s_5^{r^5(x)}t_5^{c^5(x)}.\] 
\end{definition}

\begin{proposition} \label{iso}
    If $(X,*,R_1,R_2,R_3,R_4)$ and $(Y,\triangleright,R'_1,R'_2,R'_3,R'_4)$ are isomorphic finite stuquandles, then $stqp(X)=stqp(Y)$.
\end{proposition}
\begin{proof}
    Suppose $f:X\rightarrow Y$ is a stuquandle isomorphism and fix $x\in X$. For all $y\in C^1(x)=\{y\in X| \; y*x=y\}$, we have $f(y)\triangleright f(x)=f(y*x)=f(y)$, thus $f(y)\in C^1(f(x))$ and $|C^1(x)|\leq |C^1(f(x))|$. Applying this argument to $f^{-1}$, we obtain $|C^1(f(x))|\leq |C^1(x)|$, and therefore, $c^1(x)=c^1(f(x))$. By definition of a stuquandle isomorphism we have, $R'_j(f(y),f(x))=f(R_j(y,x))=f(y)$ for $j=1,2,3,4$. By applying a similar argument used to show $c^1(x)=c^1(f(x))$, we obtain $c^i(x)=c^i(f(x))$ for $i=2,3,4,5$. A similar argument also shows that $r^i(x)=r^i(f(x))$ for $i=1,2,3,4,5$. These facts give the following:\begin{align*}
        stqp(X)&=\sum_{x\in X} s_1^{r^1(x)}t_1^{c^1(x)}s_2^{r^2(x)}t_2^{c^2(x)}s_3^{r^3(x)}t_3^{c^3(x)}s_4^{r^4(x)}t_4^{c^4(x)}s_5^{r^5(x)}t_5^{c^5(x)}\\
        &=\sum_{f(x)\in Y} s_1^{r^1(x)}t_1^{c^1(x)}s_2^{r^2(x)}t_2^{c^2(x)}s_3^{r^3(x)}t_3^{c^3(x)}s_4^{r^4(x)}t_4^{c^4(x)}s_5^{r^5(x)}t_5^{c^5(x)}\\
        &=\sum_{f(x)\in Y} s_1^{r^1(f(x))}t_1^{c^1(f(x))}s_2^{r^2(f(x))}t_2^{c^2(f(x))}s_3^{r^3(f(x))}t_3^{c^3(f(x))}s_4^{r^4(f(x))}t_4^{c^4(f(x))}s_5^{r^5(f(x))}t_5^{c^5(f(x))}\\
        &=stqp(Y).
    \end{align*}
\end{proof}

\begin{definition}
    Let $X$ be a finite stuquandle and $S\subset X$ be a substuquandle. Then the \textit{substuquandle polynomial} is 
    \[ Sstqp(S\subset X)= \sum_{x\in S} s_1^{r^1(x)}t_1^{c^1(x)}s_2^{r^2(x)}t_2^{c^2(x)}s_3^{r^3(x)}t_3^{c^3(x)}s_4^{r^4(x)}t_4^{c^4(x)}s_5^{r^5(x)}t_5^{c^5(x)}.\]
\end{definition}

Note that for $i \in \lbrace 1,2,3,4,5 \rbrace$, $r^i(x)$ (respectively $c^i(x)$) is the number of elements of $X$ that act trivially on $x$ (respectively, is the number of elements of $X$ on which $x$ acts trivially) via $*, R_1, R_2, R_3$, and $R_4$. These values can be easily computed from the operation table of $*, R_1, R_2, R_3$, and $R_4$ by counting the occurrences of the row numbers. Please refer to Example~\ref{Poly} for further explanation.
\begin{example}\label{Poly}
Let $X_1 =\mathbb{Z}_4$ be the stuquandle with operations $x*y = 3x+2y $, $R_1(x,y)=2x+3y$, $R_2(x,y) = x$, $R_3(x,y) = 3x+2y$, and $R_4(x,y) = y$. These operations have the following operation tables,
\[
\begin{tabular}{ r|  c  c  c c }
* & 0 &1 &2 & 3\\
  \hline			
0& 0 & 2 & 0 & 2 \\
1& 3 & 1 & 3 & 1 \\
2& 2 & 0 & 2 & 0 \\
3& 1 & 3 & 1 & 3
\end{tabular}
\qquad
\begin{tabular}{ r|  c  c  c c }
$R_1$ & 0 &1 &2 & 3\\
  \hline			
0& 0 & 3 & 2 & 1 \\
1& 2 & 1 & 0 & 3 \\
2& 0 & 3 & 2 & 1 \\
3& 2 & 1 & 0 & 3 \\
\end{tabular}
\qquad
\begin{tabular}{ r|  c  c  c c }
$R_2$ & 0 &1 &2 & 3\\
  \hline			
0& 0 & 0 & 0 & 0 \\
1& 1 & 1 & 1 & 1 \\
2& 2 & 2 & 2 & 2 \\
3& 3 & 3 & 3 & 3 \\
\end{tabular}
\]
\[
\begin{tabular}{ r|  c  c  c c }
$R_3$ & 0 &1 &2 & 3\\
  \hline			
0& 0 & 2 & 0 & 2 \\
1& 3 & 1 & 3 & 1 \\
2& 2 & 0 & 2 & 0 \\
3& 1 & 3 & 1 & 3
\end{tabular}
\qquad
\begin{tabular}{ r|  c  c  c c }
$R_4$ & 0 &1 &2 & 3\\
  \hline			
0& 0 & 1 & 2 & 3 \\
1& 0 & 1 & 2 & 3 \\
2& 0 & 1 & 2 & 3 \\
3& 0 & 1 & 2 & 3 \\
\end{tabular}
\]
and the operations have the following $r^i(x)$ and $c^i(x)$ values for $i= 1,2,3,4,5$:
\[
\begin{tabular}{ r|  c  c  }
$x$ & $r^1(x)$ & $c^1(x)$\\
  \hline			
1&  2 & 2 \\
2&  2 & 2  \\
3&  2 & 2 \\
0&  2 & 2  \\
\end{tabular}
\qquad
\begin{tabular}{ r|  c  c  }
$x$ & $r^2(x)$ & $c^2(x)$\\
  \hline			
1&  1 & 1 \\
2&  1 & 1  \\
3&  1 & 1 \\
0&  1 & 1  \\
\end{tabular}
\qquad
\begin{tabular}{ r|  c  c  }
$x$ & $r^3(x)$ & $c^3(x)$\\
  \hline			
1&  4& 4 \\
2&  4 & 4  \\
3&  4 & 4 \\
0&  4 & 4  \\
\end{tabular}
\]
\[
\begin{tabular}{ r|  c  c  }
$x$ & $r^4(x)$ & $c^4(x)$\\
  \hline			
1&  2 & 2 \\
2&  2 & 2  \\
3&  2 & 2 \\
0&  2 & 2  \\
\end{tabular}
\qquad
\begin{tabular}{ r|  c  c  }
$x$ & $r^5(x)$ & $c^5(x)$\\
  \hline			
1&  1 & 1 \\
2&  1 & 1  \\
3&  1 & 1 \\
0&  1 & 1  \\
\end{tabular}.
\]
Thus, the stuquandle polynomial of $X$ is 
\[ sqp(X_1) = 4 s_1^2 t_1^2 s_2 t_2 s_3^4 t_3^4 s_4^2 t_4^2 s_5 t_5. \]

Next, consider the stuquandle $X_2 = \mathbb{Z}_4$ with operations $x*y=x$, $R_1(x,y)=y$, $R_2(x,y)=x$, $R_3(x,y)=y$ and $R_4(x,y)=x$. The operations have the following operation tables
\[
\begin{tabular}{ r|  c  c  c c }
* & 0 &1 &2 & 3\\
  \hline			
0& 0 & 0 & 0 & 0 \\
1& 1 & 1 & 1 & 1 \\
2& 2 & 2 & 2 & 2 \\
3& 3 & 3 & 3 & 3 \\
\end{tabular}
\qquad
\begin{tabular}{ r|  c  c  c c }
$R_1$ & 0 &1 &2 & 3\\
  \hline			
0& 0 & 1 & 2 & 3 \\
1& 0 & 1 & 2 & 3 \\
2& 0 & 1 & 2 & 3 \\
3& 0 & 1 & 2 & 3 \\
\end{tabular}
\qquad
\begin{tabular}{ r|  c  c  c c }
$R_2$ & 0 &1 &2 & 3\\
  \hline			
0& 0 & 0 & 0 & 0 \\
1& 1 & 1 & 1 & 1 \\
2& 2 & 2 & 2 & 2 \\
3& 3 & 3 & 3 & 3 \\
\end{tabular}
\]
\[
\begin{tabular}{ r|  c  c  c c }
$R_3$ & 0 &1 &2 & 3\\
  \hline			
0& 0 & 1 & 2 & 3 \\
1& 0 & 1 & 2 & 3 \\
2& 0 & 1 & 2 & 3 \\
3& 0 & 1 & 2 & 3 \\
\end{tabular}
\qquad
\begin{tabular}{ r|  c  c  c c }
$R_4$ & 0 &1 &2 & 3\\
  \hline			
0& 0 & 0 & 0 & 0 \\
1& 1 & 1 & 1 & 1 \\
2& 2 & 2 & 2 & 2 \\
3& 3 & 3 & 3 & 3 \\
\end{tabular}
\]
and the operations have the following $r^i(x)$ and $c^i(x)$ values for $i= 1,2,3,4,5$:
\[
\begin{tabular}{ r|  c  c  }
$x$ & $r^1(x)$ & $c^1(x)$\\
  \hline			
1&  4 & 4 \\
2&  4 & 4  \\
3&  4 & 4 \\
0&  4 & 4  \\
\end{tabular}
\qquad
\begin{tabular}{ r|  c  c  }
$x$ & $r^2(x)$ & $c^2(x)$\\
  \hline			
1&  1 & 1 \\
2&  1 & 1  \\
3&  1 & 1 \\
0&  1 & 1  \\
\end{tabular}
\qquad
\begin{tabular}{ r|  c  c  }
$x$ & $r^3(x)$ & $c^3(x)$\\
  \hline			
1&  4& 4 \\
2&  4 & 4  \\
3&  4 & 4 \\
0&  4 & 4  \\
\end{tabular}
\]
\[
\begin{tabular}{ r|  c  c  }
$x$ & $r^4(x)$ & $c^4(x)$\\
  \hline			
1&  1 & 1 \\
2&  1 & 1  \\
3&  1 & 1 \\
0&  1 & 1  \\
\end{tabular}
\qquad
\begin{tabular}{ r|  c  c  }
$x$ & $r^5(x)$ & $c^5(x)$\\
  \hline			
1&  4 & 4 \\
2&  4 & 4  \\
3&  4 & 4 \\
0&  4 & 4  \\
\end{tabular}.
\]
Thus, the stuquandle polynomial of $X$ is 
\[ sqp(X_2) = 4 s_1^4 t_1^4 s_2 t_2 s_3^4 t_3^4 s_4^1 t_4^1 s_5^4 t_5^4. \]

We obtain that $sqp(X_1) \neq sqp(X_2)$. Therefore, the two stuquandle structures defined on $\mathbb{Z}_4$ are distinguished by the contrapositive of Proposition~\ref{iso}. 
\end{example}

\begin{example}
    Let $S=\{1,3\}$ be a substuquandle of $X$ from the previous example. Thus, the substuquandle polynomial of $S$ is 
    \[ Sstqp(S\subset X)= 2 s_1^2 t_1^2 s_2 t_2 s_3^4 t_3^4 s_4^2 t_4^2 s_5 t_5.\]
\end{example}

By Lemma~\ref{image}, we know that the image of a stuquandle homomorphism is a substuquandle. Suppose that \(X\) is a stuquandle and \(L\) is a stuck link. For each \(f \in \textup{Hom}(\mathcal{STQ}(L), X)\), the image \(\text{Im}(f)\) is a substuquandle of \(X\). This allows us to define the following polynomial.

\begin{definition}
    Let $L$ be a stuck link, $T$ be a finite stuquandle. Then the multiset 
    \[ \Phi _{Sstqp}(L,T)=\{Sstqp(Im(f)\subset T)|f\in Hom(\mathcal{STQ}(L),T)\}\]
    is the \textit{substuquandle polynomial invariant of $L$} with respect to $T$. We can rewrite the multiset in a polynomial-style form by converting the multiset elements to exponents of a formal variable $u$ and converting their multiplicities to coefficients:
\[ \phi_{Sstqp}(L,T) = \sum_{f \in \textup{Hom}(\mathcal{STQ}(L),T)} u^{Sstqp(Im(f) \subset T)}.\]
\end{definition}

\section{Examples}\label{Examples}

In this section, we present examples that demonstrate the effectiveness of the substuquandle polynomial invariant in distinguishing stuck links. Specifically, we include an example that illustrates how the substuquandle polynomial invariant enhances the stuquandle counting invariant. Additionally, we include an example of two stuck links that are not distinguished by the $\hat{X}$ polynomial, but can be distinguished by the substuquandle polynomial. Finally, we explicitly compute the substuquandle polynomial of two RNA foldings and differentiate them using the substuquandle polynomial.

\begin{example}\label{Example71}
     Consider coloring the stuck knots $2_1^{k-}$ and $0_1^{k+}$ by the stuquandle $X=\mathbb{Z}_{4}$ with operation defined by $x*y=3x+2y$ and maps $R_1(x,y)=x+2y^2$, $R_2(x,y)=2x^2+y$, $R_3(x,y)=3x$, and $R_4(x,y)=2x+y$. It was found in \cite{CEMR} that the stuquandle counting invariant for both of these knots is equal to 4. \begin{figure}[h]
    \begin{minipage}{0.48\textwidth}
    \centering
\includegraphics{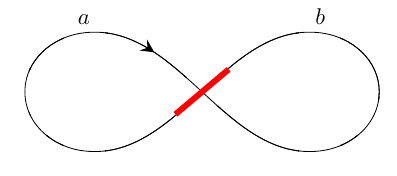}
\caption{ The stuck knot infinity $0^{k+}_1$}
\label{stuckinfinity}
    \end{minipage}
    \begin{minipage}{0.48\textwidth}
    \centering
\includegraphics{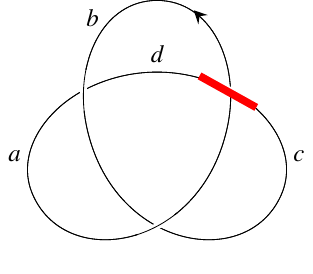}
    \caption{ The stuck trefoil $2^{k-}_1$}
    \label{stucktrefoil}
    \end{minipage}
\end{figure}
Given the stuquandle defined above, the colorings for Figure \ref{stuckinfinity} are $\textup{Hom}(\mathcal{STQ}(0_1^{k+}), X)  = \lbrace(0,0),(0,2),(2,0),(2,2) \rbrace$, and the colorings for Figure \ref{stucktrefoil} are $\textup{Hom}(\mathcal{STQ}(2_1^{k-}), X)  = \lbrace(0, 0, 0, 0)$, $(1, 3, 3, 1), (2, 2, 2, 2), (3, 1, 1, 3) \rbrace$. Thus, the coloring invariant for both knots is $4$. To calculate the substuquandle polynomial invariant of these stuck knots, we show in the tables below the operations have the following $r^i(x)$ and $c^i(x)$ values for $i= 1,2,3,4,5$:
\[
\begin{tabular}{ r|  c  c  }
$x$ & $r^1(x)$ & $c^1(x)$\\
  \hline
0&  2 & 2  \\
1&  2 & 2 \\
2&  2 & 2  \\
3&  2 & 2 \\
\end{tabular}
\qquad
\begin{tabular}{ r|  c  c  }
$x$ & $r^2(x)$ & $c^2(x)$\\
  \hline		
0&  2 & 4  \\
1&  2 & 0 \\
2&  2 & 4  \\
3&  2 & 0 \\
\end{tabular}
\qquad
\begin{tabular}{ r|  c  c  }
$x$ & $r^3(x)$ & $c^3(x)$\\
  \hline	
0&  1 & 1  \\
1&  1& 1 \\
2&  1 & 1  \\
3&  1 & 1 \\
\end{tabular}
\]
\[
\begin{tabular}{ r|  c  c  }
$x$ & $r^4(x)$ & $c^4(x)$\\
  \hline	
0&  4 & 2  \\
1&  0 & 2 \\
2&  4 & 2  \\
3&  0 & 2 \\
\end{tabular}
\qquad
\begin{tabular}{ r|  c  c  }
$x$ & $r^5(x)$ & $c^5(x)$\\
  \hline	
0&  1 & 1  \\
1&  1 & 1 \\
2&  1 & 1  \\
3&  1 & 1 \\
\end{tabular}.
\]
We collect the colorings of each stuck link and the substuquandle polynomial of the image of each coloring in Table~\ref{poly1} and Table~\ref{poly2}. 
\begin{table}[h]
{\begin{tabular}{cc|c|c}
$f(a)$ & $f(b)$ &   $Im(f) \subset X$ &$Sstqp(Im(f)\subset X)$ \\
\hline
 $0$& $0$ &  $\{ 0 \}$ & $s_1^2 t_1^2 s_2^2 t_2^4 s_3 t_3 s_4^4 t_4^2 s_5 t_5$\\
 $0$& $2$ &  $\{ 0,2 \}$ & $2 s_1^2 t_1^2 s_2^2 t_2^4 s_3 t_3 s_4^4 t_4^2 s_5 t_5$ \\
  $2$& $0$ &  $\{ 0,2 \}$ & $2 s_1^2 t_1^2 s_2^2 t_2^4 s_3 t_3 s_4^4 t_4^2 s_5 t_5$\\
   $2$& $2$ &  $\{ 2 \}$ & $s_1^2 t_1^2 s_2^2 t_2^4 s_3 t_3 s_4^4 t_4^2 s_5 t_5$\\
\end{tabular}}
\caption{Colorings and substuquandle polynomial of each coloring of $0^{k+}_1$.}
\label{poly1}
\end{table}
\begin{table}[h]
{\begin{tabular}{cccc|c|c}
$f(a)$ & $f(b)$ & $f(c)$ & $f(d)$ &  $Im(f) \subset X$ &$Sstqp(Im(f)\subset X)$ \\
\hline
 $0$& $0$ & $0$& $0$ &  $\{ 0 \}$ &$s_1^2 t_1^2 s_2^2 t_2^4 s_3 t_3 s_4^4 t_4^2 s_5 t_5$\\
 $1$& $3$ & $3$ &  $1$ &  $\{ 1,3 \}$ &$2s_1^2 t_1^2 s_2^2 s_3 t_3  t_4^2 s_5 t_5$\\
  $2$& $2$ & $2$& $2$ &  $\{ 2 \}$ &$s_1^2 t_1^2 s_2^2 t_2^4 s_3 t_3 s_4^4 t_4^2 s_5 t_5$\\
   $3$& $1$ & $1$ & $3$ &  $\{ 1,3 \}$ &$2s_1^2 t_1^2 s_2^2 s_3 t_3  t_4^2 s_5 t_5$\\
\end{tabular}}
\caption{Colorings and substuquandle polynomial of each coloring of $2^{k-}_1$.}
\label{poly2}
\end{table}

Using the coloring set of each stuck knot, we obtain the following substuquandle polynomial invariants
\[ \phi _{Sstqp}(0_1^{k+},X)=2u^{2 s_1^2 t_1^2 s_2^2 t_2^4 s_3 t_3 s_4^4 t_4^2 s_5 t_5}+2u^{ s_1^2 t_1^2 s_2^2 t_2^4 s_3 t_3 s_4^4 t_4^2 s_5 t_5}
\]
and

\[\phi _{Sstqp}(2_1^{k-},X)=2u^{ s_1^2 t_1^2 s_2^2 t_2^4 s_3 t_3 s_4^4 t_4^2 s_5 t_5}+2u^{2s_1^2 t_1^2 s_2^2 s_3 t_3  t_4^2 s_5 t_5}.\]

\begin{figure}[h]
    \begin{minipage}{0.48\textwidth}
    \centering
\includegraphics[scale=1]{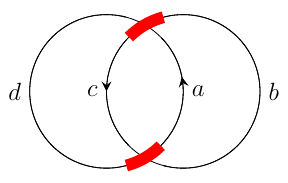}
\caption{Diagram of stuck link $K_1$.}
\label{link1}
    \end{minipage}
    \begin{minipage}{0.48\textwidth}
    \centering
    \includegraphics[scale=1]{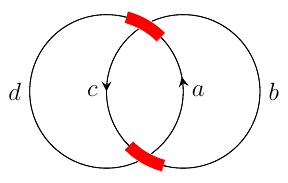}
    \caption{Diagram of stuck link $K_2$.}
    \label{link}
    \end{minipage}
\end{figure}
\end{example}

\begin{example}
 Consider coloring the stuck knots $K_1$ and $K_2$ in Figure~\ref{link1} and \ref{link}, respectively, using the stuquandle $X=\mathbb{Z}_{3}$ with operations $x*y = x$, $R_1(x,y)= 2y^2$,$R_2(x,y)=2x^2$, $R_3(x,y) = 2x+2x^2$, and $R_4(x,y)=2y+2y^2$. The operation tables are given by

 \[
\begin{tabular}{ r|  c  c   c }
* & 0 &1 & 2\\
  \hline			
0& 0 & 0 & 0  \\
1& 1 & 1 & 1    \\
2& 2 & 2 & 2  \\
\end{tabular}
\qquad
\begin{tabular}{ r|  c  c   c }
$R_1$ & 0 &1  & 2\\
  \hline			
0&0 & 2 & 2  \\
1& 0 & 2 & 2  \\
2& 0 & 2 & 2  \\
\end{tabular}
\qquad
\begin{tabular}{ r|  c  c  c  }
$R_2$ & 0 &1 & 2\\
  \hline			
 0&0 & 0 & 0  \\
 1&2 & 2 & 2  \\
 2&2 & 2 & 2  \\
\end{tabular}
\]
\[
\begin{tabular}{ r|  c  c   c }
$R_3$ & 0 &1  & 2\\
  \hline			
 0&0 & 0 & 0  \\
 1&1 & 1 & 1 \\
 2&0 & 0 & 0  \\
\end{tabular}
\qquad
\begin{tabular}{ r|  c  c   c }
$R_4$ & 0 &1  & 2\\
  \hline			
 0&0 & 1 & 0  \\
 1&0 & 1 & 0  \\
 2&0 & 1 & 0  \\
\end{tabular}
\]
 The $\hat{X}$ polynomial was defined in \cite{B} for stuck knots and it was shown that $\hat{X}(K_1) = 
 2xy-(x^2+y^2)(A^2+A{-2}) = \hat{X}(K_2)$. Also note 
 that $\textup{Hom}(\mathcal{STQ}(K_1), X)= \lbrace(0,0,0,0), 
 (0,1,0,1), (1,0,1,0), (1,1,1,1) \rbrace $ and 
 $\textup{Hom}(\mathcal{STQ}(K_1), X)= \lbrace (0,0,0,0), 
 (0,2,0,2), (2,0,2,0), (2,2,2,2) \rbrace $. Thus, the 
 coloring invariant for both is $4$. To calculate the substuquandle polynomial invariant of these stuck knots, we show in the tables below the operations have the following $r^i(x)$ and $c^i(x)$ values for $i= 1,2,3,4,5$:
\[
\begin{tabular}{ r|  c  c  }
$x$ & $r^1(x)$ & $c^1(x)$\\
  \hline
0&  3 & 3  \\
1&  3 & 3 \\
2&  3 & 3  \\

\end{tabular}
\qquad
\begin{tabular}{ r|  c  c  }
$x$ & $r^2(x)$ & $c^2(x)$\\
  \hline		
0&  1 & 1  \\
1&  0 & 1 \\
2&  2 & 1  \\

\end{tabular}
\qquad
\begin{tabular}{ r|  c  c  }
$x$ & $r^3(x)$ & $c^3(x)$\\
  \hline	
0&  3 & 2  \\
1&  0& 2 \\
2&  3 & 2  \\

\end{tabular}
\]
\[
\begin{tabular}{ r|  c  c  }
$x$ & $r^4(x)$ & $c^4(x)$\\
  \hline	
0&  3 & 2  \\
1&  3 & 2 \\
2&  0 & 2  \\

\end{tabular}
\qquad
\begin{tabular}{ r|  c  c  }
$x$ & $r^5(x)$ & $c^5(x)$\\
  \hline	
0&  2 & 1  \\
1&  1 & 1 \\
2&  0 & 1  \\

\end{tabular}.
\]
We collect the colorings of each stuck link and the substuquandle polynomial of the image of each coloring in Table~\ref{poly3} and Table~\ref{poly4}. 
\begin{table}[h]
{\begin{tabular}{cccc|c|c}
$f(a)$ & $f(b)$ & $f(c)$ & $f(d)$ &  $Im(f) \subset X$ &$Sstqp(Im(f)\subset X)$ \\
\hline
 $0$& $0$ & $0$& $0$ &  $\{ 0 \}$ &$s_1^3 t_1^3 s_2 t_2 s_3^3 t_3^2 s_4^3 t_4^2 s_5^2 t_5$\\
 $1$& $0$ & $1$ &  $0$ &  $\{ 0,1,2 \}$ &$s_1^3 t_1^3 s_2 t_2 s_3^3 t_3^2 s_4^3 t_4^2 s_5^2 t_5+s_1^3 t_1^3 t_2 t_3^2 s_4^3 t_4^2 s_5 t_5+s_1^3 t_1^3 s_2 t_2 s_3^3 t_3^2 t_4^2 t_5$\\
  $0$& $1$ & $0$& $1$ &  $\{ 0,1,2 \}$ &$s_1^3 t_1^3 s_2 t_2 s_3^3 t_3^2 s_4^3 t_4^2 s_5^2 t_5+s_1^3 t_1^3 t_2 t_3^2 s_4^3 t_4^2 s_5 t_5+s_1^3 t_1^3 s_2 t_2 s_3^3 t_3^2 t_4^2 t_5$\\
   $1$& $1$ & $1$ & $1$ &  $\{ 0,1,2 \}$ &$s_1^3 t_1^3 s_2 t_2 s_3^3 t_3^2 s_4^3 t_4^2 s_5^2 t_5+s_1^3 t_1^3 t_2 t_3^2 s_4^3 t_4^2 s_5 t_5+s_1^3 t_1^3 s_2 t_2 s_3^3 t_3^2 t_4^2 t_5$\\
\end{tabular}}
\caption{Colorings and substuquandle polynomial of each coloring of $K_1$.}
\label{poly3}
\end{table}
\begin{table}[h]
{\begin{tabular}{cccc|c|c}
$f(a)$ & $f(b)$ & $f(c)$ & $f(d)$ &  $Im(f) \subset X$ &$Sstqp(Im(f)\subset X)$ \\
\hline
 $0$& $0$ & $0$& $0$ &  $\{ 0 \}$ &$s_1^3 t_1^3 s_2 t_2 s_3^3 t_3^2 s_4^3 t_4^2 s_5^2 t_5$\\
 $0$& $2$ & $0$ &  $2$ &  $\{ 0,2 \}$ &$s_1^3 t_1^3 s_2 t_2 s_3^3 t_3^2 s_4^3 t_4^2 s_5^2 t_5+s_1^3 t_1^3 s_2 t_2 s_3^3 t_3^2 t_4^2 t_5 $\\
  $2$& $0$ & $2$& $0$ &  $\{ 0,2 \}$ &$s_1^3 t_1^3 s_2 t_2 s_3^3 t_3^2 s_4^3 t_4^2 s_5^2 t_5+s_1^3 t_1^3 s_2 t_2 s_3^3 t_3^2 t_4^2 t_5 $\\
   $2$& $2$ & $2$ & $2$ &  $\{ 0,2 \}$ &$s_1^3 t_1^3 s_2 t_2 s_3^3 t_3^2 s_4^3 t_4^2 s_5^2 t_5+s_1^3 t_1^3 s_2 t_2 s_3^3 t_3^2 t_4^2 t_5 $\\
\end{tabular}}
\caption{Colorings and substuquandle polynomial of each coloring of $2^{k-}_1$.}
\label{poly4}
\end{table}

Using the coloring set of each stuck link, we obtain the following substuquandle polynomial invariants
\[ \phi _{Sstqp}(K_1,X)=u^{s_1^3 t_1^3 s_2 t_2 s_3^3 t_3^2 s_4^3 t_4^2 s_5^2 t_5}+3u^{ s_1^3 t_1^3 s_2 t_2 s_3^3 t_3^2 s_4^3 t_4^2 s_5^2 t_5+s_1^3 t_1^3 t_2 t_3^2 s_4^3 t_4^2 s_5 t_5+s_1^3 t_1^3 s_2 t_2 s_3^3 t_3^2 t_4^2 t_5}
\]
and

\[\phi _{Sstqp}(K_2,X)=u^{s_1^3 t_1^3 s_2 t_2 s_3^3 t_3^2 s_4^3 t_4^2 s_5^2 t_5}+3u^{s_1^3 t_1^3 s_2 t_2 s_3^3 t_3^2 s_4^3 t_4^2 s_5^2 t_5+s_1^3 t_1^3 s_2 t_2 s_3^3 t_3^2 t_4^2 t_5},\]
 this shows that our invariant is stronger than the $\hat{X}$ polynomial of \cite{B}.   
\end{example}

Now we will compute our new substuquandle polynomial invariant to distinguish the topology of RNA structure. First, we will consider an arc diagram of RNA folding, then use the transformation $T$ and apply the self-closure to obtain a stuck link diagram. We will then compute the substuquandle polynomial invariant using the stuck link diagram corresponding to the arc diagram of an RNA folding. Note that since the substuquandle polynomial invariant is unchanged by Reidemeister, the invariant only depends on the stuck link and not the diagram. 

\begin{example}
 Let $(\mathbb{Z}_4,*,R_1,R_2,R_3,R_4)$ be the stuquandle with operations $x*y =x$, 
 $R_1(x,y)= 3 x + y$,$R_2(x,y)=x + 3 y$, $R_3(x,y) = x + 2 y$, and $R_4(x,y)=2x+y$. 
 The operation tables are given by

 \[
\begin{tabular}{ r|  c  c   c c}
* & 0 &1 & 2 &3\\
\hline			
0& 0& 0& 0& 0\\
1& 1& 1& 1& 1\\ 
2& 2& 2& 2& 2\\
3& 3& 3& 3& 3
\end{tabular}
\qquad
\begin{tabular}{ r|  c  c   cc }
$R_1$ & 0 &1  & 2 & 3\\
  \hline			
0& 0& 1& 2& 3\\
1& 3& 0& 1& 2\\ 
2& 2& 3& 0& 1\\
3& 1& 2& 3& 0
\end{tabular}
\qquad
\begin{tabular}{ r|  c  c  c c }
$R_2$ & 0 &1 & 2 & 3\\
  \hline			
0& 0& 3& 2& 1\\
1& 1& 0& 3& 2\\ 
2& 2& 1& 0& 3\\
3& 3& 2& 1& 0
\end{tabular}
\]
\[
\begin{tabular}{ r|  c  c   c c}
$R_3$ & 0 &1  & 2 &3\\
  \hline			
0& 0& 2& 0& 2\\ 
1& 1& 3& 1& 3\\ 
2& 2& 0& 2& 0\\ 
3& 3& 1& 3& 1
\end{tabular}
\qquad
\begin{tabular}{ r|  c  c   c c}
$R_4$ & 0 &1  & 2 & 3\\
  \hline			
0& 0& 1& 2& 3\\ 
1& 2& 3& 0& 1\\ 
2& 0& 1& 2& 3\\ 
3& 2& 3& 0& 1
\end{tabular}
\]
We will consider the two arc diagrams of RNA foldings and their corresponding stuck links; see Figure~\ref{fig:arc1} and Figure~\ref{fig:arc2}.

\begin{figure}
    \centering
    \includegraphics[scale=.5]{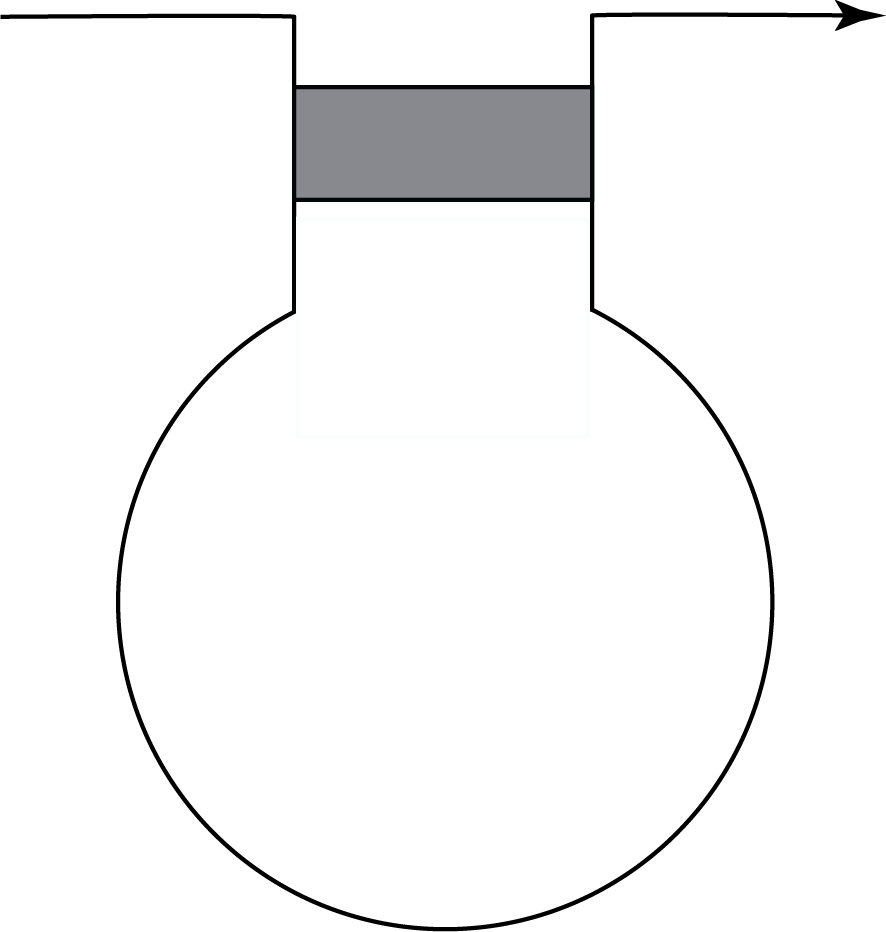}
    \hspace{1.5cm}
    \includegraphics[scale=.5]{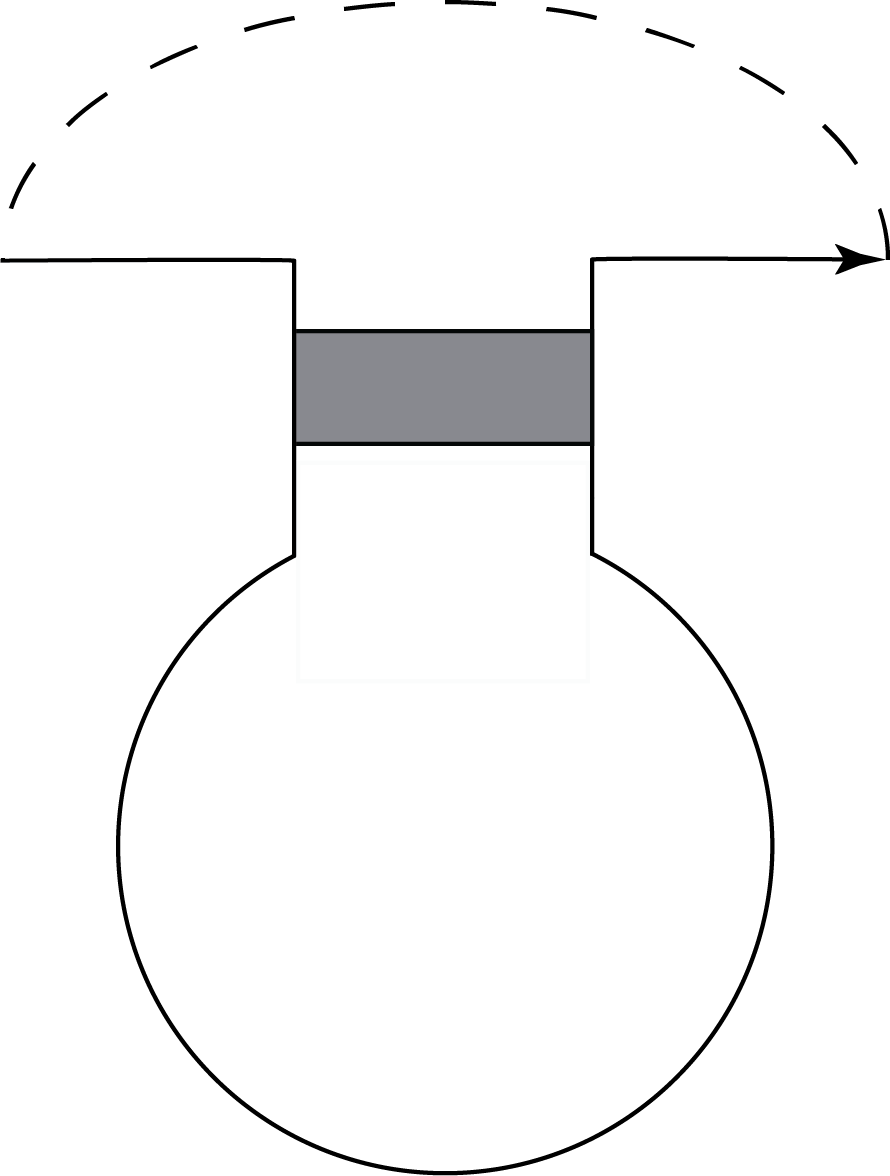}
    \put(-160,50){self-closure}
    \includegraphics{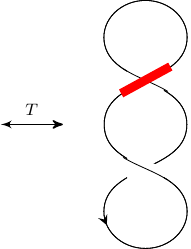}
    \put(-50,108){$a$}
    \put(-50,50){$b$}
    \put(0,50){$c$}
    \caption{Arc diagram and corresponding stuck link denoted by $K_1$.}
    \label{fig:arc1}
\end{figure}

\begin{figure}
    \centering
    \includegraphics[scale=.5]{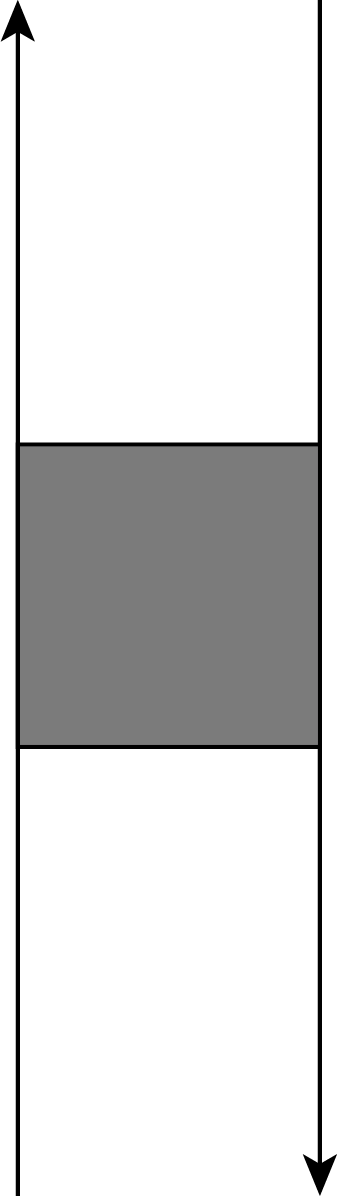}
    \hspace{2.5cm}
    \includegraphics[scale=.5]{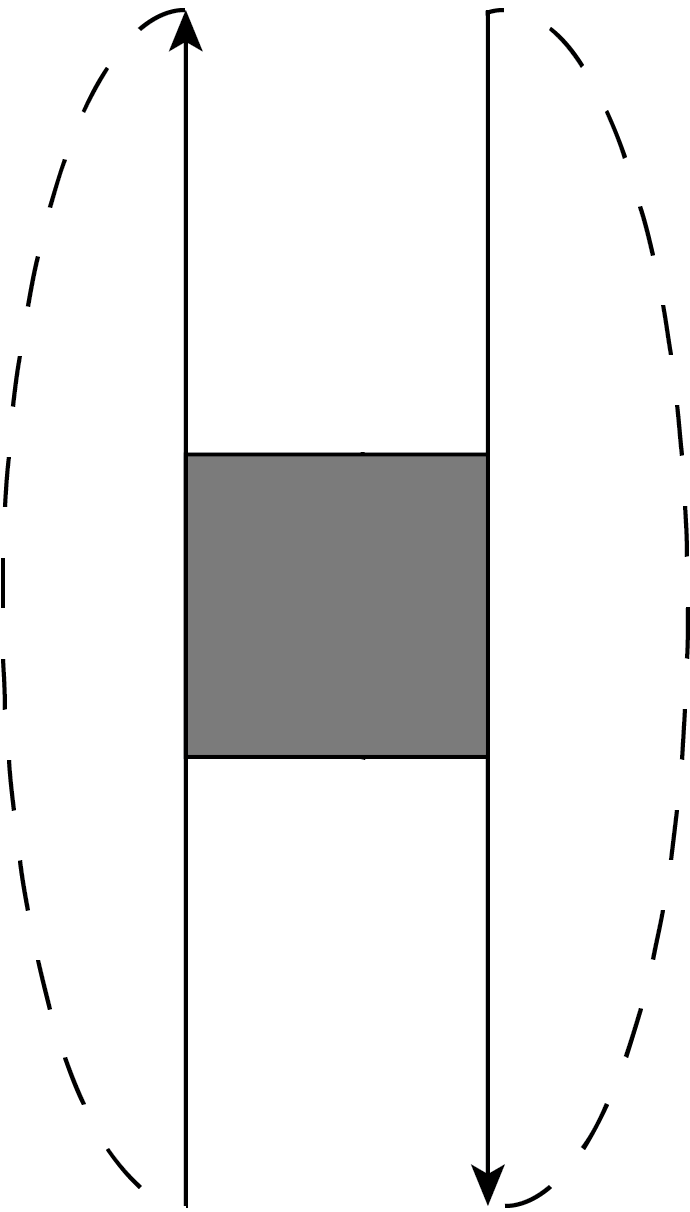}
    \put(-150,60){self-closure}
    \hspace{.5cm}
\includegraphics{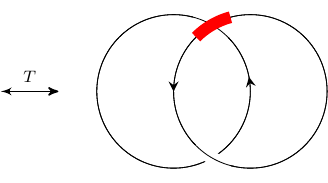}
\put(-30,50){$a$}
    \put(-25,80){$b$}
    \put(-120,50){$c$}
    \caption{Arc diagram and corresponding stuck link denoted by $K_2$.}
    \label{fig:arc2}
\end{figure}

With respect to the stuquandle defined above, the colorings for Figure~\ref{fig:arc1} are $\hom(\mathcal{STQ}(K_1),X) = \lbrace (0,0,0), (1,3,3), (2,2,2), (3,1,1)\rbrace$ and the colorings for Figure~\ref{fig:arc2} are 
$\hom(\mathcal{STQ}(K_2),X) = \lbrace (0,0,0), (0,2,0), (2,0,2), (2,2,2)\rbrace$.  In this case, the stuquandle counting invariant cannot distinguish the two arc diagrams. Now, we will consider the substuquandle polynomial invariant of the two arc diagrams. In the tables below, we collect our operations $r^i$ and $c^i$ values: 
\[
\begin{tabular}{ r|  c  c  }
$x$ & $r^1(x)$ & $c^1(x)$\\
  \hline
0&  4 & 4  \\
1&  4 & 4 \\
2&  4 & 4  \\
3&  4 & 4  

\end{tabular}
\qquad
\begin{tabular}{ r|  c  c  }
$x$ & $r^2(x)$ & $c^2(x)$\\
  \hline		
0&  1 & 2  \\
1&  1 & 0 \\
2&  1 & 2  \\
3&  1 & 0  
\end{tabular}
\qquad
\begin{tabular}{ r|  c  c  }
$x$ & $r^3(x)$ & $c^3(x)$\\
  \hline	
0&  1 & 4  \\
1&  1 & 0 \\
2&  1 & 0  \\
3&  1 & 0  

\end{tabular}
\]
\[
\begin{tabular}{ r|  c  c  }
$x$ & $r^4(x)$ & $c^4(x)$\\
  \hline	
0&  2 & 4  \\
1&  2 & 0 \\
2&  2 & 4  \\
3&  2 & 0

\end{tabular}
\qquad
\begin{tabular}{ r|  c  c  }
$x$ & $r^5(x)$ & $c^5(x)$\\
  \hline	
0&  1 & 1  \\
1&  1 & 1 \\
2&  1 & 1  \\
3&  1 & 1
\end{tabular}.
\]
We collect the colorings of each arc diagram and the substuquandle polynomial of the image of each coloring in Table~\ref{arcpoly1} and Table~\ref{arcpoly2}. 
\begin{table}[h]
{\begin{tabular}{ccc|c|c}
$f(a)$ & $f(b)$ & $f(c)$ &  $Im(f) \subset X$ &$Sstqp(Im(f)\subset X)$ \\
\hline
 $0$& $0$ & $0$ &  $\{ 0 \}$ & $s_1^4 t_1^4  s_2  t_2^2  s_3  t_3^4  s_4^2  t_4^4  s_5  t_5$\\
 $1$& $3$ & $3$  &  $\{ 0,1,2,3 \}$ &$2 s_1^4  t_1^4  s_2  s_3  s_4^2  s_5  t_5 + s_1^4  t_1^4  s2  t_2^2  s_3  t_3^4  s_4^2  t_4^4  s_5  t_5 + s_1^4  t_1^4  s_2  t_2^2  s_3  s_4^2  t_4^4  s_5  t_5$\\
  $2$& $2$ & $2$ &  $\{ 0,2 \}$ &$s_1^4  t_1^4  s2  t_2^2  s_3  t_3^4  s_4^2  t_4^4  s_5  t_5 + s_1^4  t_1^4  s_2  t_2^2  s_3  s_4^2  t_4^4  s_5  t_5$\\
   $3$& $1$ & $1$ &  $\{ 0,1,2,3 \}$ &$2 s_1^4  t_1^4  s_2  s_3  s_4^2  s_5  t_5 + s_1^4  t_1^4  s2  t_2^2  s_3  t_3^4  s_4^2  t_4^4  s_5  t_5 + s_1^4  t_1^4  s_2  t_2^2  s_3  s_4^2  t_4^4  s_5  t_5$\\
\end{tabular}}
\caption{Colorings and substuquandle polynomial of each coloring of $K_1$.}
\label{arcpoly1}
\end{table}
\begin{table}[h]
{\begin{tabular}{ccc|c|c}
$f(a)$ & $f(b)$ & $f(c)$ &  $Im(f) \subset X$ &$Sstqp(Im(f)\subset X)$ \\
\hline
 $0$& $0$ & $0$ &  $\{ 0 \}$ & $s_1^4 t_1^4  s_2  t_2^2  s_3  t_3^4  s_4^2  t_4^4  s_5  t_5$\\
 $0$& $2$ & $0$  &  $\{ 0,2 \}$ & $s_1^4  t_1^4  s2  t_2^2  s_3  t_3^4  s_4^2  t_4^4  s_5  t_5 + s_1^4  t_1^4  s_2  t_2^2  s_3  s_4^2  t_4^4  s_5  t_5$\\
  $2$& $0$ & $2$ &  $\{ 0,2 \}$ & $s_1^4  t_1^4  s2  t_2^2  s_3  t_3^4  s_4^2  t_4^4  s_5  t_5 + s_1^4  t_1^4  s_2  t_2^2  s_3  s_4^2  t_4^4  s_5  t_5$\\
   $2$& $2$ & $2$ &  $\{ 0,2 \}$ &$s_1^4  t_1^4  s2  t_2^2  s_3  t_3^4  s_4^2  t_4^4  s_5  t_5 + s_1^4  t_1^4  s_2  t_2^2  s_3  s_4^2  t_4^4  s_5  t_5$\\
\end{tabular}}
\caption{Colorings and substuquandle polynomial of each coloring of $K_2$.}
\label{arcpoly2}
\end{table}
Using the coloring set of each arc diagram, we obtain the following substuquandle polynomial invariants

\begin{eqnarray*}
\phi _{Sstqp}\begin{pmatrix}\includegraphics[scale=.2]{Figure15a},X\end{pmatrix}=&&u^{s_1^4 t_1^4  s_2  t_2^2  s_3  t_3^4  s_4^2  t_4^4  s_5  t_5}+ u^{s_1^4  t_1^4  s2  t_2^2  s_3  t_3^4  s_4^2  t_4^4  s_5  t_5 + s_1^4  t_1^4  s_2  t_2^2  s_3  s_4^2  t_4^4  s_5  t_5}\\&+&2u^{ 2 s_1^4  t_1^4  s_2  s_3  s_4^2  s_5  t_5 + s_1^4  t_1^4  s2  t_2^2  s_3  t_3^4  s_4^2  t_4^4  s_5  t_5 + s_1^4  t_1^4  s_2  t_2^2  s_3  s_4^2  t_4^4  s_5  t_5}
\end{eqnarray*}
and
\[
\phi _{Sstqp}\begin{pmatrix}\includegraphics[scale=.2]{Figure16a},X\end{pmatrix}=u^{s_1^4 t_1^4  s_2  t_2^2  s_3  t_3^4  s_4^2  t_4^4  s_5  t_5}+3u^{s_1^4  t_1^4  s2  t_2^2  s_3  t_3^4  s_4^2  t_4^4  s_5  t_5 + s_1^4  t_1^4  s_2  t_2^2  s_3  s_4^2  t_4^4  s_5  t_5}.\]

Thus, the RNA foldings are distinguished since 
\[
\phi _{Sstqp}\begin{pmatrix}\includegraphics[scale=.2]{Figure15a},X\end{pmatrix} \neq
\phi _{Sstqp}\begin{pmatrix}\includegraphics[scale=.2]{Figure16a},X\end{pmatrix}.
\]

\end{example}

\section*{Acknowledgement} 
Mohamed Elhamdadi was partially supported by Simons Foundation collaboration grant 712462.




\bibliography{Ref}
\bibliographystyle{plain}

\end{document}